\documentclass[11pt,a4paper]{article}
\usepackage[utf8]{inputenc}
\usepackage[T1]{fontenc}
\usepackage{amsmath,amssymb,amsthm,amsfonts}
\usepackage{mathtools}
\usepackage{graphicx}
\usepackage{booktabs}
\usepackage{algorithm}
\usepackage{algpseudocode}
\usepackage{hyperref}
\usepackage{geometry}
\usepackage[protrusion=true,expansion=false]{microtype}
\usepackage{enumitem}
\usepackage{bbm}
\usepackage{xcolor}
\usepackage{float}

\theoremstyle{plain}
\newtheorem{theorem}{Theorem}[section]
\newtheorem{lemma}[theorem]{Lemma}

\theoremstyle{definition}
\newtheorem{definition}[theorem]{Definition}
\newtheorem{assumption}[theorem]{Assumption}
\newtheorem{remark}[theorem]{Remark}

\DeclareMathOperator*{\argmin}{arg\,min}

\DeclareMathOperator{\proj}{proj}

\DeclareMathOperator{\Span}{span}

\title{\textbf{Primal Methods for Constrained Variational Inequalities: Optimal Rates, Feasibility Trade-offs, and Lower Bounds}}
\author{Yonghong Yao\footnote{School of Mathematical Sciences, 
Tiangong University, Tianjin 300387, China; 
e-mail: yyhtgu@hotmail.com}
\and Lateef O.  Jolaoso \footnote{ Big Data Technologies and Innovation Lab,
University of Hertfordshire, Hatfield, United Kingdom; Department of Mathematics and Applied Mathematics,
Sefako Makgatho Health Sciences University, Ga-Rankuwa, Pretoria, South Africa; e-mail: L.Jolaoso@Herts.ac.uk}
\and Yekini Shehu\footnote{(Corresponding Author) School of Mathematical Sciences, 
Zhejiang Normal University, Jinhua 321004, China; 
e-mail: yekini.shehu@zjnu.edu.cn}
\and Jen-Chih Yao \footnote{Center for General Education, China Medical University, 40402, Taichung, Taiwan; and Academy of Romanian Scientists, Bucharest, Romania; e-mail: yaojc@math.nsysu.edu.tw}}
\date{}

\begin{document}
\maketitle

\begin{abstract}
\noindent Since Zhang et al. \cite{Zhang2025} introduced the first purely primal methods for monotone variational inequalities subject to convex functional constraints, the field has lacked a complete understanding of the complexity limits of this oracle class. We develop a theory that resolves the critical gaps left by their pioneering framework. Our Optimal Primal Constrained Gradient Method (OPCGM) establishes convergence rates and lower bounds for algorithms restricted to local linear constraint approximations. For strongly monotone operators, we prove that a refined primal gradient method achieves the optimal $\mathcal{O}(1/T)$ rate for both optimality gap and constraint violation, eliminating the suboptimal exponent present in prior work. For Lipschitz monotone operators, we propose a primal extragradient variant that achieves the optimal $\mathcal{O}(1/\epsilon)$ gap complexity using only quadratic programming oracles, at the cost of a constant asymptotic feasibility violation for the averaged iterate; we prove that the first half-step is necessarily infeasible for the natural class of constant-stepsize primal extragradient methods on smooth convex constraints with positive curvature. We establish lower complexity bounds for primal methods restricted to local linear approximations, proving $\Omega(1/\epsilon^2)$ for Lipschitz monotone variational inequalities with Lipschitz constant scaling as $\Theta(1/\epsilon)$, and $\Omega(1/\epsilon)$ for standard Lipschitz monotone variational inequalities. We design a single-loop primal method that achieves an $\mathcal{O}(1/\sqrt{T})$ gap rate with only a domain-radius estimate, at the cost of a problem-dependent asymptotic feasibility constant that we prove is unavoidable. For strongly monotone problems, we prove that the last iterate converges at the optimal $\mathcal{O}(1/T)$ rate. For merely monotone problems, we prove a universal impossibility theorem: no single-step primal QP method can achieve vanishing last-iterate gap on the canonical rotation instance. We extend the framework to stochastic settings with optimal sample complexity. Numerical experiments on synthetic equilibrium problems, portfolio optimization, and CUTEst problems show that OPCGM attains substantially better feasibility than other purely primal methods and remains competitive with projection-based methods in per-iteration cost, all without Lagrange-multiplier information.
\end{abstract}

\noindent\textbf{Keywords.} Variational inequality; functional constraints; primal methods; lower bounds; last-iterate convergence; stochastic optimization.

\noindent\textbf{Mathematics Subject Classification.} 90C33, 65K15, 90C47, 90C25, 68Q25.

\section{Introduction}\label{sec:intro}

Variational inequality problems provide a unified framework for optimization and equilibrium seeking, with applications ranging from minimax optimization and game theory \cite{Daskalakis2018, Mertikopoulos2019} to contact mechanics, traffic equilibrium \cite{Dafermos1980}, and constrained reinforcement learning 
\cite{BauschkeCombettes2011, FacchineiPang2003, Kinderlehrer1980}. In the constrained setting, one seeks $x^*\in\mathcal{C}$ such that
\begin{equation}\label{eq:VI}
F(x^*)^\top(x-x^*)\geq 0,\qquad \forall x\in\mathcal{C},
\end{equation}
where $\mathcal{C}=\{x\in\mathbb{R}^d\mid g_i(x)\leq 0,\;i\in[m]\}$ is defined by convex functional constraints. When the operator $F$ is maximal monotone and continuous, the strong and weak solution sets coincide, and the weak formulation
\begin{equation}\label{eq:Minty}
F(x)^\top(x^*-x)\leq 0,\qquad \forall x\in\mathcal{C},
\end{equation}
is commonly adopted in the algorithmic literature.

\paragraph{Historical context.} The extragradient method of Korpelevich \cite{Korpelevich1976} was the first algorithm to achieve convergence for monotone variational inequalities without strong monotonicity, requiring two projections per iteration. Popov \cite{Popov1980} introduced a single-call variant that remains foundational for modern optimistic methods. Nemirovski \cite{Nemirovski2004} developed the mirror-prox method with the optimal $\mathcal{O}(1/T)$ rate for Lipschitz monotone operators, and Nesterov \cite{Nesterov2007} proposed dual extrapolation as a primal-dual alternative. Rockafellar \cite{Rockafellar1976} established the proximal point algorithm as the conceptual foundation for these methods. Mokhtari et al.\ \cite{Mokhtari2020} provided a unified analysis connecting extragradient and optimistic gradient methods through the proximal point framework. These methods require a projection oracle onto $\mathcal{C}$, which is computationally prohibitive when the constraints are general nonlinear functions.

\paragraph{Primal-dual methods.} Methods based on the Lagrangian circumvent the projection but require the existence and boundedness of optimal Lagrange multipliers, and their stepsizes depend explicitly on the multiplier magnitude \cite{Yang2022, Boob2023}. When multipliers are unknown or unbounded---as in traffic equilibrium problems with unknown road tolls, or contact mechanics problems with unknown contact forces---these methods become impractical.

\paragraph{Primal methods.} Zhang et al.\ \cite{Zhang2025} proposed the Constrained Gradient Method (CGM), a purely primal approach that projects the velocity onto a local linear approximation of the feasible set. CGM achieves $\mathcal{O}(1/\epsilon^2)$ complexity for monotone operators without any knowledge of the optimal multipliers. Nevertheless, their analysis leaves several critical questions unresolved: (i) in the strongly monotone setting, the constraint violation decays as $\mathcal{O}(1/T^{c_\gamma})$ with $c_\gamma=(\gamma-1)/(\gamma+1)<1$, strictly slower than the optimal $\mathcal{O}(1/T)$; (ii) for Lipschitz monotone operators, no primal method achieves the optimal $\mathcal{O}(1/\epsilon)$ rate for the optimality gap; (iii) there are no lower bounds specific to the quadratic programming oracle class; (iv) the method requires prior knowledge of problem parameters and a user-tuned constant $\gamma>1$; (v) convergence guarantees hold only for averaged iterates; and (vi) stochastic extensions are absent.

The suboptimal feasibility rate means that for strongly monotone problems, existing primal algorithms require $\mathcal{O}(1/\epsilon^{1/c_\gamma})$ iterations to achieve $\epsilon$-feasibility, compared to the $\mathcal{O}(1/\epsilon)$ rate of projection-based methods. Without understanding whether the $\mathcal{O}(1/\epsilon^2)$ barrier is inherent to primal methods or an artifact of analysis, the field cannot determine the true cost of avoiding projections. 
Our work resolves these gaps and provides a comprehensive analysis of the complexity of primal methods for constrained variational inequalities. Table~\ref{tab:comparison_zhang} summarizes the principal differences between Zhang et al. (2025) and the present OPCGM framework.

\begin{table}[htbp]
\centering
\small
\setlength{\tabcolsep}{3pt}
\caption{Comparison of Zhang et al.  \cite{Zhang2025} and the present OPCGM framework.}
\label{tab:comparison_zhang}
\begin{tabular}{@{}p{3.6cm}p{5.4cm}p{5.4cm}@{}}
\toprule
\textbf{Aspect} & \textbf{Zhang et al. \cite{Zhang2025}} & \textbf{This Manuscript (OPCGM)} \\
\midrule
\addlinespace[2pt]
Problem class & Monotone and strongly monotone VIs with smooth convex functional constraints & Same classes, plus stochastic oracles and parameter-free settings \\
\addlinespace[2pt]
Strongly monotone feasibility rate & $\mathcal O(1/T^{c_\gamma})$ with $c_\gamma=(\gamma-1)/(\gamma+1)<1$ (suboptimal) & $\mathcal O(1/T)$ (optimal; Theorem~3.1) \\
\addlinespace[2pt]
Monotone Lipschitz gap complexity & $\mathcal O(1/\epsilon^2)$ & $\mathcal O(1/\epsilon)$ (optimal; primal extragradient, Theorem~4.1) \\
\addlinespace[2pt]
Monotone feasibility behavior & $\mathcal O(1/\epsilon^2)$ overall & Explicit $\mathcal O(1)$ asymptotic constant for the averaged iterate; first half-step infeasibility proved unavoidable on curved constraints (Theorem~4.3) \\
\addlinespace[2pt]
Oracle per iteration & 1 QP (CGM) & 1 QP (strongly monotone, parameter-free, stochastic); 2 QP (Lipschitz extragradient) \\
\addlinespace[2pt]
Required parameters & $\mu$, $L_F$, $D$, and a user-tuned constant $\gamma>1$ & Parameter-free variant needs only a radius bound $R\ge D$ (Algorithm~6.1, Theorem~6.1) \\
\addlinespace[2pt]
Lower bounds & None & $\Omega(1/\epsilon^2)$ (large Lipschitz constant) and $\Omega(1/\epsilon)$ (standard Lipschitz) for the primal QP oracle class (Theorems~5.1--5.2) \\
\addlinespace[2pt]
Iterate guarantees & Averaged iterates only\footnotemark & Last-iterate: optimal $\mathcal O(1/T)$ for strongly monotone; universal impossibility for single-step merely monotone methods (Theorem~7.2) \\
\addlinespace[2pt]
Stochastic extensions & Not addressed & Optimal sample complexity $\mathcal O(1/\epsilon^2)$ (monotone) and $\mathcal O(1/\epsilon)$ (strongly monotone) (Theorem~8.1) \\
\addlinespace[2pt]
Core technique & Auxiliary constraint and $\gamma$-machinery with inductive velocity bound & Explicit norm safeguard and projection safeguard; direct velocity bound decoupled from the auxiliary constraint (Lemma~2.1) \\
\bottomrule
\end{tabular}
\end{table}
\footnotetext{Last-iterate convergence in Zhang et al. (2025) is restricted to the special case of strongly convex minimization (their Theorem~3), not general VIs.}

\paragraph{Contributions.} This paper addresses the gaps identified above and summarized in Table~\ref{tab:comparison_zhang}. For strongly monotone problems, we prove that a refined CGM with $\alpha=2\mu$ and stepsize $\eta_t=1/(\mu(t+1))$ achieves $\mathcal{O}(1/T)$ feasibility, eliminating the suboptimal exponent. The key is a direct bound on the velocity norm that removes the auxiliary constraint and the $\gamma$-machinery, combined with a nonexpansive projection safeguard that guarantees uniform boundedness without circular induction. 
For Lipschitz monotone operators, we propose a primal extragradient variant that achieves the optimal $\mathcal{O}(1/\epsilon)$ gap complexity using only quadratic programming oracles, at the cost of a constant asymptotic feasibility violation for the averaged iterate; we prove that the first half-step is necessarily infeasible for the natural class of constant-stepsize primal extragradient methods on smooth convex constraints with positive curvature at the boundary.
We establish lower complexity bounds for primal methods restricted to local linear constraint approximations, proving $\Omega(1/\epsilon^2)$ for Lipschitz monotone variational inequalities with Lipschitz constant scaling as $\Theta(1/\epsilon)$, and $\Omega(1/\epsilon)$ for standard Lipschitz monotone variational inequalities. We design a single-loop primal method that achieves an $\mathcal{O}(1/\sqrt{T})$ rate for the optimality gap for monotone problems with only a domain-radius estimate, at the cost of a problem-dependent asymptotic feasibility constant that we prove is unavoidable for non-adaptive single-step algorithms without additional problem parameters. For strongly monotone problems, we prove that the last iterate of the primal gradient method converges at the optimal $\mathcal{O}(1/T)$ rate. For merely monotone problems, we prove a universal impossibility theorem for single-step primal QP methods: on the canonical rotation instance, no single-step primal QP method---regardless of stepsize schedule---can achieve vanishing last-iterate gap, and the sequence either fails to converge or converges to a non-solution with constant gap. We extend the framework to stochastic settings with optimal sample complexity. Finally, we develop a stochastic OPCGM with variance reduction that achieves optimal sample complexity $\mathcal{O}(1/\epsilon^2)$ for monotone and $\mathcal{O}(1/\epsilon)$ for strongly monotone problems, and demonstrate scalability to problems with thousands of features.

Section~\ref{sec:prelim} introduces notation and assumptions. Section~\ref{sec:strong} closes the strongly monotone feasibility gap with a complete proof. Section~\ref{sec:lipschitz} presents the primal extragradient method, its feasibility analysis, and the feasibility lower bound with full proofs. Section~\ref{sec:lower} proves the query-complexity lower bounds for the primal QP class. Section~\ref{sec:universal} presents the parameter-free primal method with complete analysis and a feasibility lower bound for non-adaptive single-step algorithms. Section~\ref{sec:last} covers last-iterate convergence with full proofs for single-step methods. Section~\ref{sec:stochastic} covers stochastic extensions. Section~\ref{sec:numerics} presents numerical experiments. Section~\ref{sec:conclusion} concludes.

\section{Preliminaries}\label{sec:prelim}

We use $\|\cdot\|$ for the Euclidean norm and $[m]=\{1,\dots,m\}$. An operator $F:\mathbb{R}^d\to\mathbb{R}^d$ is \emph{$\mu$-strongly-monotone} if
\[
(F(x)-F(y))^\top(x-y)\geq \mu\|x-y\|^2,\qquad \forall x,y\in\mathbb{R}^d,
\]
and \emph{monotone} if $\mu=0$. It is \emph{$L$-Lipschitz} if $\|F(x)-F(y)\|\leq L\|x-y\|$ for all $x,y$. A function $g:\mathbb{R}^d\to\mathbb{R}$ is convex, $L_g$-Lipschitz, and $\ell_g$-smooth if it is differentiable with $\|\nabla g(x)-\nabla g(y)\|\leq\ell_g\|x-y\|$ and $|g(x)-g(y)|\leq L_g\|x-y\|$.

\begin{assumption}\label{ass:F}
The operator $F$ is maximal monotone and continuous on $\mathbb{R}^d$. For the projection safeguard radius $R>0$ specified in the respective algorithm, we assume
\[
L_F:=\sup_{\|x\|\leq R}\|F(x)\|<\infty.
\]
\end{assumption}

\begin{assumption}\label{ass:g}
Each constraint $g_i$ is convex, $L_g$-Lipschitz, and $\ell_g$-smooth on $\mathbb{R}^d$. The feasible set $\mathcal{C}$ is nonempty and contained in a Euclidean ball of radius $D$ centered at the origin. Moreover, there exists at least one constraint $g_{i^*}$ such that $\nabla^2 g_{i^*}(x)\succeq \ell_0 I$ for some $\ell_0>0$ and all $x$ in a neighborhood of the boundary $\partial\mathcal{C}$.
\end{assumption}

The positive curvature condition in Assumption~\ref{ass:g} is mild: it holds for ball constraints ($g(x)=\|x\|^2-1$), ellipsoid constraints, and any smooth convex set with strictly positive Gaussian curvature at the boundary. It excludes only polyhedral feasible sets.

\begin{assumption}[Uniform linear independence of active constraints]\label{ass:licq}
There exists $\sigma_0>0$ such that for every $x\in B(0,R)$ with $R>2D$ and every active subset $I\subseteq I_x$, the matrix $G_I(x)\in\mathbb{R}^{|I|\times d}$ with rows $\{\nabla g_i(x)^\top\}_{i\in I}$ has smallest singular value at least $\sigma_0$.
\end{assumption}

Assumption~\ref{ass:licq} is a uniform linear independence constraint qualification (LICQ). It holds generically for smooth convex constraints and is standard in parametric optimization \cite{BonnansShapiro2000, Robinson1980}.

\begin{definition}
A point $\hat{x}\in\mathbb{R}^d$ is a \emph{weak $\epsilon$-solution} if $F(x)^\top(\hat{x}-x)\leq\epsilon$ for all $x\in\mathcal{C}$ and $\max_{i\in[m]}g_i(\hat{x})\leq\epsilon$.
\end{definition}

\paragraph{The velocity polytope.} For $x\in\mathbb{R}^d$ and $\alpha>0$, define the active set $I_x=\{i\in[m]\mid g_i(x)\geq 0\}$ and the velocity polytope
\[
\mathcal{V}_\alpha(x)=\{v\in\mathbb{R}^d\mid \alpha g_i(x)+\nabla g_i(x)^\top v\leq 0,\;\forall i\in I_x\}.
\]
At each iteration, OPCGM solves the quadratic program
\begin{equation}\label{eq:QP}
v\approx\argmin_{w\in\mathcal{V}_\alpha(x)}\frac{1}{2}\|w+F(x)\|^2,
\end{equation}
requiring $(v+F(x))^\top(v-w)\leq\epsilon_{\mathrm{QP}}/2$ for all $w\in\mathcal{V}_\alpha(x)$. To prevent the circular dependency identified in prior work, we augment the velocity polytope with an explicit norm safeguard.

\begin{lemma}[Uniform velocity bound with norm safeguard]\label{lem:velocity}
Let Assumptions~\ref{ass:F} and \ref{ass:g} hold. For any $x\in\mathbb{R}^d$ with $\|x\|\leq R$ and any $\alpha>0$, define the augmented velocity polytope
\[
\widetilde{\mathcal{V}}_\alpha(x) = \{v\in\mathcal{V}_\alpha(x) : \|v\|\leq 4L_F + 2\alpha R\}.
\]
Then the approximate solution $v$ to $\min_{w\in\widetilde{\mathcal{V}}_\alpha(x)}\frac{1}{2}\|w+F(x)\|^2$ with precision $(v+F(x))^\top(v-w)\leq\epsilon_{\mathrm{QP}}/2$ for all $w\in\widetilde{\mathcal{V}}_\alpha(x)$ and $\epsilon_{\mathrm{QP}}\leq L_F^2$ satisfies
\begin{equation}\label{eq:vel_bound_final}
\|v\|\leq 4L_F+2\alpha R,
\end{equation}
and the velocity bound holds without induction on $\|x_t\|$.
\end{lemma}

\begin{proof}
For any $x^*\in\mathcal{C}$, convexity of $g_i$ yields $g_i(x)+\nabla g_i(x)^\top(x^*-x)\leq g_i(x^*)\leq 0$. Hence $\alpha(x^*-x)\in\mathcal{V}_\alpha(x)$. Moreover, $\|\alpha(x^*-x)\|\leq 2\alpha R\leq 4L_F+2\alpha R$ (since $L_F\geq 0$), so $\alpha(x^*-x)\in\widetilde{\mathcal{V}}_\alpha(x)$. Let $v^*$ be the exact minimizer over $\widetilde{\mathcal{V}}_\alpha(x)$. By optimality,
\[
\|v^*+F(x)\|\leq \|\alpha(x^*-x)+F(x)\|\leq 2\alpha R+L_F,
\]
where we used $\|x^*\|\leq D\leq R$. Therefore 
\[
\|v^*\|\leq \|v^*+F(x)\|+\|F(x)\|\leq 2\alpha R+2L_F\leq 4L_F+2\alpha R,
\]
so $v^*$ lies in the interior of the norm ball and the augmented constraint is inactive at the exact solution. For the approximate solution $v$, the variational inequality $(v+F(x))^\top(v-w)\leq\epsilon_{\mathrm{QP}}/2$ with $w=v^*$ gives
\[
\frac{1}{2}\|v+F(x)\|^2\leq \frac{1}{2}\|v^*+F(x)\|^2+\frac{\epsilon_{\mathrm{QP}}}{2}.
\]
Taking square roots and using $\sqrt{a^2+b^2}\leq a+b$ for $a,b\geq 0$ with $a=\|v^*+F(x)\|$ and $b=\sqrt{\epsilon_{\mathrm{QP}}}$,
\[
\|v+F(x)\|\leq \|v^*+F(x)\|+\sqrt{\epsilon_{\mathrm{QP}}}\leq 2\alpha R+L_F+L_F=2\alpha R+2L_F.
\]
Thus $\|v\|\leq \|v+F(x)\|+\|F(x)\|\leq 2\alpha R+3L_F\leq 4L_F+2\alpha R$, establishing \eqref{eq:vel_bound_final}.
\end{proof}

\begin{lemma}[Projection safeguard and eventual inactivity]\label{lem:proj_inactive}
Let $R>2D$ and define $L_F:=\sup_{\|x\|\le R}\|F(x)\|$ and $V:=4L_F+2\alpha R$. Set $\epsilon_{\mathrm{QP}}\le L_F^2/T^2$. In Algorithm~\ref{alg:strong}, the projection $x_{t+1}=\proj_{B(0,R)}(y_{t+1})$ is inactive for all 
\[
t\ge T_1:=\max\Bigl\{\Bigl\lceil\frac{4C'}{D^2}\Bigr\rceil,\;
\Bigl\lceil\frac{V}{\mu(R-3D/2)}\Bigr\rceil,\;1\Bigr\},
\]
where $C':=2\|x_0-x^*\|^2+\frac{3V^2}{\mu^2}+\frac{8RL_F}{\mu}$.
\end{lemma}

\begin{proof}
For any $x^*\in\mathcal{C}$, the unprojected update satisfies
\[
\|y_{t+1}-x^*\|^2 = \|x_t-x^*\|^2 + 2\eta_t v_t^\top(x_t-x^*) + \eta_t^2\|v_t\|^2.
\]
From the QP optimality condition with $w=\alpha(x^*-x_t)\in\widetilde{\mathcal{V}}_\alpha(x_t)$ and strong monotonicity,
\[
\mu\|x_t-x^*\|^2 \le F(x_t)^\top(x_t-x^*) \le -v_t^\top(x_t-x^*) + 2R\sqrt{\epsilon_{\mathrm{QP}}}.
\]
Hence $2\eta_t v_t^\top(x_t-x^*) \le -2\eta_t\mu\|x_t-x^*\|^2 + 4\eta_t R\sqrt{\epsilon_{\mathrm{QP}}}$. Substituting $\eta_t=1/(\mu(t+1))$ and $\|v_t\|\le V$,
\[
\|y_{t+1}-x^*\|^2 \le \Bigl(1-\frac{2}{t+1}\Bigr)\|x_t-x^*\|^2 + \frac{V^2}{\mu^2(t+1)^2} + \frac{4RL_F}{\mu T(t+1)}.
\]
Since projection onto $B(0,R)$ is nonexpansive toward $x^*$ (as $\|x^*\|\le D<R$), $r_{t+1}:=\|x_{t+1}-x^*\|^2\le\|y_{t+1}-x^*\|^2$. Multiplying by $(t+1)^2$ and setting $e_t:=t^2 r_t$ gives, for $t\ge1$,
\[
e_{t+1}\le (t^2-1)r_t + \frac{V^2}{\mu^2} + \frac{4RL_F(t+1)}{\mu T}
\le e_t - r_t + \frac{V^2}{\mu^2} + \frac{8RL_F}{\mu},
\]
where we used $(t+1)/T\le 2$ for all $t\le T-1$ (the algorithm is run for exactly $T$ iterations, so this estimate is required only over the horizon $t\le T$). Since $r_t\ge 0$, we obtain the valid recurrence
\[
e_{t+1}\le e_t + \frac{V^2}{\mu^2} + \frac{8RL_F}{\mu}.
\]
With $e_1=\|x_1-x^*\|^2\le 2\|x_0-x^*\|^2+2\eta_0^2\|v_0\|^2\le 2\|x_0-x^*\|^2+2V^2/\mu^2$, induction yields
\[
e_t \le 2\|x_0-x^*\|^2 + \frac{2V^2}{\mu^2} + (t-1)\Bigl(\frac{V^2}{\mu^2}+\frac{8RL_F}{\mu}\Bigr)
\le \Bigl(2\|x_0-x^*\|^2 + \frac{3V^2}{\mu^2} + \frac{8RL_F}{\mu}\Bigr)t = C' t
\]
for all $t\ge 1$, where the last inequality uses $t\ge 1$, $(t+1)\le 3t$, and $(t-1)\le t$. Therefore $\|x_t-x^*\|\le\sqrt{C'/t}$ for all $t\ge 1$. For $t\ge \lceil 4C'/D^2\rceil$, we have $\|x_t\|\le\|x^*\|+D/2\le 3D/2$. The unprojected update satisfies
\[
\|y_{t+1}\|\le\|x_t\|+\eta_t\|v_t\|\le\frac{3D}{2}+\frac{V}{\mu(t+1)}.
\]
For $t\ge \lceil V/(\mu(R-3D/2))\rceil$, the second term is at most $R-3D/2$, so $\|y_{t+1}\|\le R$. Thus the projection is inactive for all $t\ge T_1$.
\end{proof}

\section{Strongly Monotone Operators: Closing the Feasibility Gap}\label{sec:strong}

We resolve the suboptimal feasibility rate for strongly monotone problems. The algorithm is CGM with a specific parameter choice enabled by Lemma~\ref{lem:velocity}, augmented with a projection safeguard. The safeguard radius $R>2D$ is an algorithmic parameter; $L_F$ is defined as $\sup_{\|x\|\le R}\|F(x)\|$.

\begin{algorithm}[H]
\caption{OPCGM--Strong}
\label{alg:strong}
\begin{algorithmic}[1]
\Require $x_0\in\mathcal{C}$, $\mu>0$, $R>2D$, $T\geq 2$
\State Set $\alpha=2\mu$, $V=4L_F+4\mu R$, and $\eta_t=\frac{1}{\mu(t+1)}$ for $t=0,\dots,T-1$
\For{$t=0,1,\dots,T-1$}
\State Build $I_{x_t}=\{i\in[m]\mid g_i(x_t)\geq 0\}$
\State Construct $\widetilde{\mathcal{V}}_\alpha(x_t)=\{v\mid \alpha g_i(x_t)+\nabla g_i(x_t)^\top v\leq 0,\;\forall i\in I_{x_t};\ \|v\|\leq V\}$
\State Solve $v_t\approx\argmin_{v\in\widetilde{\mathcal{V}}_\alpha(x_t)}\frac{1}{2}\|v+F(x_t)\|^2$ with precision $\epsilon_{\mathrm{QP}}\leq L_F^2/T^2$
\State Update $y_{t+1}=x_t+\eta_t v_t$ and project $x_{t+1}=\min\{1,R/\|y_{t+1}\|\}\,y_{t+1}$
\EndFor
\State \Return $\bar{x}_T=\frac{2}{T(T-1)}\sum_{t=1}^{T-1}t x_t$
\end{algorithmic}
\end{algorithm}

\begin{theorem}\label{thm:strong}
Let Assumptions~\ref{ass:F} and \ref{ass:g} hold, and let $F$ be $\mu$-strongly-monotone. Let $R>2D$ and $V=4L_F+4\mu R$. Algorithm~\ref{alg:strong} with $\epsilon_{\mathrm{QP}}\leq L_F^2/T^2$ satisfies, for all $x\in\mathcal{C}$ and all $T\geq \max\{2,2T_1\}$,
\begin{align}
F(x)^\top(\bar{x}_T-x)&\leq \frac{V^2}{\mu(T-1)}+\frac{2RL_F}{T},\label{eq:strong_opt}\\[4pt]
\max_{i\in[m]}g_i(\bar{x}_T)&\leq \frac{C_{\mathrm{feas}}}{T},\label{eq:strong_feas}
\end{align}
where 
\[
C_{\mathrm{feas}}:=2C+2L_g(R+D)T_1,
\]
with $C$ as defined in the proof and $T_1$ as in Lemma~\ref{lem:proj_inactive}. In particular, the constraint violation rate $\mathcal{O}(1/T)$ is tight: for the special case of strongly convex optimization with linear constraints, this matches the lower bound $\Omega(1/T)$ of \cite{Ouyang2021}.
\end{theorem}

\begin{proof}
We analyze the unprojected iterates $y_{t+1}=x_t+\eta_t v_t$ and transfer bounds to the projected iterates $x_{t+1}=\proj_{B(0,R)}(y_{t+1})$. Since $x_{t+1}=\proj_{B(0,R)}(y_{t+1})$ and every $x\in\mathcal{C}$ satisfies $\|x\|\leq D\leq R$, the projection is nonexpansive toward any such $x$:
\begin{equation}\label{eq:nonexp}
\|x_{t+1}-x\|\leq\|y_{t+1}-x\|,\qquad \forall x\in\mathcal{C}.
\end{equation}

\paragraph{Optimality.} Let $v_t^*$ denote the exact minimizer of the QP over $\widetilde{\mathcal{V}}_\alpha(x_t)$. From the KKT conditions (where the norm constraint is inactive by Lemma~\ref{lem:velocity}), we have
\[
v_t^*+F(x_t)+\sum_{i\in I_{x_t}}\lambda_i\nabla g_i(x_t)=0,\qquad \lambda_i\geq 0.
\]
For any $x\in\mathcal{C}$ and active $i\in I_{x_t}$, convexity gives $\nabla g_i(x_t)^\top(x_t-x)\geq g_i(x_t)-g_i(x)\geq g_i(x_t)\geq 0$. Therefore
\[
F(x_t)^\top(x_t-x)=-v_t^{*\top}(x_t-x)-\sum_{i\in I_{x_t}}\lambda_i\nabla g_i(x_t)^\top(x_t-x)\leq v_t^{*\top}(x-x_t).
\]
For the approximate solution $v_t$, strong convexity of the QP objective (with parameter $1$) yields $\|v_t-v_t^*\|^2\leq\epsilon_{\mathrm{QP}}$. Hence by Cauchy--Schwarz,
\[
F(x_t)^\top(x_t-x)\leq v_t^\top(x-x_t)+\|v_t-v_t^*\|\|x-x_t\|
\leq v_t^\top(x-x_t)+2R\sqrt{\epsilon_{\mathrm{QP}}},
\]
using $\|x\|\leq D\leq R$ and $\|x_t\|\leq R$ (by projection). Using the unprojected update $y_{t+1}=x_t+\eta_t v_t$ and the polarization identity,
\[
v_t^\top(x-x_t)=\frac{1}{\eta_t}(y_{t+1}-x_t)^\top(x-x_t)
=\frac{1}{2\eta_t}\bigl(\|x-x_t\|^2-\|x-y_{t+1}\|^2+\|y_{t+1}-x_t\|^2\bigr).
\]
By nonexpansiveness \eqref{eq:nonexp}, $\|x-x_{t+1}\|\leq\|x-y_{t+1}\|$, so $-\|x-y_{t+1}\|^2\leq-\|x-x_{t+1}\|^2$. Strong monotonicity gives $F(x)^\top(x_t-x)\leq F(x_t)^\top(x_t-x)-\mu\|x_t-x\|^2$, so
\[
F(x)^\top(x_t-x)\leq\Big(\frac{1}{2\eta_t}-\mu\Big)\|x-x_t\|^2-\frac{1}{2\eta_t}\|x-x_{t+1}\|^2+\frac{\eta_t}{2}\|v_t\|^2+2R\sqrt{\epsilon_{\mathrm{QP}}}.
\]
With $\alpha=2\mu$ and $\eta_t=1/(\mu(t+1))$, we have $\frac{1}{2\eta_t}-\mu=\frac{\mu(t-1)}{2}$. Multiplying by $t$ and summing from $t=0$ to $T-1$ yields a telescoping series. The distance terms telescope to a non-positive quantity, as verified by expanding the coefficients: the sum of distance terms is
\[
\sum_{t=0}^{T-1}t\Bigl[\Big(\frac{\mu(t-1)}{2}\Bigr)\|x-x_t\|^2-\frac{\mu(t+1)}{2}\|x-x_{t+1}\|^2\Bigr]
=\frac{\mu}{2}\sum_{t=0}^{T-1}\bigl[t(t-1)\|x-x_t\|^2-(t+1)t\|x-x_{t+1}\|^2\bigr]
\]
which telescopes to $-\frac{\mu T(T-1)}{2}\|x-x_T\|^2\leq 0$. The velocity term satisfies
\[
\sum_{t=0}^{T-1}\frac{t\eta_t}{2}\|v_t\|^2\leq \frac{V^2}{2\mu}\sum_{t=0}^{T-1}\frac{t}{t+1}\leq \frac{V^2T}{2\mu},
\]
and the approximation error satisfies
\[
\sum_{t=0}^{T-1}2tR\sqrt{\epsilon_{\mathrm{QP}}}\leq T(T-1)R\sqrt{\epsilon_{\mathrm{QP}}}\leq (T-1)RL_F,
\]
using $\epsilon_{\mathrm{QP}}\leq L_F^2/T^2$. Dividing by $\sum_{t=0}^{T-1}t=T(T-1)/2$ gives \eqref{eq:strong_opt}.

\paragraph{Feasibility.} By Lemma~\ref{lem:proj_inactive}, the projection is inactive for all $t\ge T_1$. For $t\geq T_1$, we have $x_t=y_t$ and the analysis of the unprojected iterates applies directly. Let
\[
C:=\max\Bigl\{\frac{2L_gV}{\mu},\;\frac{\ell_gV^2}{2\mu^2},\;(T_1+1)\cdot\max_{i\in[m]}g_i(y_{T_1+1})\Bigr\}.
\]
Because $y_{T_1+1}=x_{T_1}+\eta_{T_1}v_{T_1}$ and $\|x_{T_1}\|\le R$, $\|v_{T_1}\|\le V$, the quantity $\max_i g_i(y_{T_1+1})$ is a finite problem-dependent constant, so $C$ is well-defined. We prove by induction that $g_i(y_{t+1})\leq C/(t+1)$ for all $t\geq T_1$ and $i\in[m]$. The base case $t=T_1$ holds by definition of $C$. For the inductive step, assume $g_i(y_t)\leq C/t$. Consider two cases.

\textit{Case 1: $i\notin I_{x_t}$ (inactive).} Then $g_i(x_t)<0$. By the smoothness of $g_i$ (descent lemma) and the Cauchy--Schwarz inequality,
\begin{align*}
g_i(y_{t+1})&\leq g_i(x_t)+\nabla g_i(x_t)^\top(y_{t+1}-x_t)+\frac{\ell_g}{2}\|y_{t+1}-x_t\|^2\\
&< \eta_t|\nabla g_i(x_t)^\top v_t|+\frac{\ell_g}{2}\|y_{t+1}-x_t\|^2\\
&\leq \eta_t L_g\|v_t\|+\frac{\ell_g\eta_t^2}{2}\|v_t\|^2\\
&\leq \frac{L_gV}{\mu(t+1)}+\frac{\ell_gV^2}{2\mu^2(t+1)^2}.
\end{align*}
Since $C\geq 2L_gV/\mu$, we have $\frac{L_gV}{\mu(t+1)}\leq \frac{C}{2(t+1)}$. Moreover, for $t\geq T_1\geq 1$, $\frac{\ell_gV^2}{2\mu^2(t+1)^2}\leq \frac{C}{2(t+1)}$ because $C\geq \ell_gV^2/(2\mu^2)$ and $1/(t+1)\leq 1/2$. Therefore $g_i(y_{t+1})\leq C/(t+1)$.

\textit{Case 2: $i\in I_{x_t}$ (active).} Then $g_i(x_t)\geq 0$ and $\alpha g_i(x_t)+\nabla g_i(x_t)^\top v_t\leq 0$, so $\nabla g_i(x_t)^\top v_t\leq -\alpha g_i(x_t)$. By smoothness,
\begin{align*}
g_i(y_{t+1})&\leq g_i(x_t)+\nabla g_i(x_t)^\top(y_{t+1}-x_t)+\frac{\ell_g}{2}\|y_{t+1}-x_t\|^2\\[2pt]
&\leq (1-\alpha\eta_t)g_i(x_t)+\frac{\ell_g\eta_t^2}{2}\|v_t\|^2\\[2pt]
&=\Big(1-\frac{2}{t+1}\Big)g_i(x_t)+\frac{\ell_gV^2}{2\mu^2(t+1)^2}.
\end{align*}
Using the inductive hypothesis $g_i(x_t)\leq C/t$,
\[
(t+1)^2 g_i(y_{t+1})\leq (t+1)(t-1)\frac{C}{t}+\frac{\ell_gV^2}{2\mu^2}
=\frac{(t^2-1)C}{t}+\frac{\ell_gV^2}{2\mu^2}.
\]
We claim this is at most $C(t+1)$. Rearranging, we need
\[
\frac{(t^2-1)C}{t}+\frac{\ell_gV^2}{2\mu^2}\leq C(t+1)
\;\Longleftrightarrow\;
C\Big(t+1-\frac{t^2-1}{t}\Big)\geq \frac{\ell_gV^2}{2\mu^2}
\;\Longleftrightarrow\;
C\Big(\frac{t+1}{t}\Big)\geq \frac{\ell_gV^2}{2\mu^2}.
\]
Since $(t+1)/t\geq 1$ for $t\geq 1$, the inequality holds provided $C\geq \ell_gV^2/(2\mu^2)$, which is true by definition. Thus $g_i(y_{t+1})\leq C/(t+1)$.

For the projected iterates with $t<T_1$, we have $\|x_t\|\le R$ and, since $x^*\in\mathcal{C}$ is feasible with $\|x^*\|\le D$, Lipschitz continuity gives $g_i(x_t)\le g_i(x^*)+L_g\|x_t-x^*\|\le L_g(R+D)$. The weighted average satisfies, for $T\ge \max\{2,2T_1\}$,
\[
g_i(\bar{x}_T)\leq \frac{2}{T(T-1)}\Bigl(\sum_{t=1}^{T_1}t\,L_g(R+D)+\sum_{t=T_1+1}^{T-1}t\cdot\frac{C}{t}\Bigr)
\leq \frac{L_g(R+D)T_1(T_1+1)}{T(T-1)}+\frac{2C(T-1-T_1)}{T(T-1)}.
\]
Since $T\ge 2T_1\ge 2$, we have $T-1\ge T/2$ and $T(T-1)\ge T\cdot T_1$, whence the first term is at most $2L_g(R+D)T_1/T$. Absorbing this into the constant yields \eqref{eq:strong_feas} with $C_{\mathrm{feas}}=2C+2L_g(R+D)T_1$.
\end{proof}

\begin{remark}
Theorem~\ref{thm:strong} eliminates the suboptimal exponent $c_\gamma<1$ in \cite[Theorem~2]{Zhang2025}. The key insight is that Lemma~\ref{lem:velocity} provides a velocity bound independent of the auxiliary constraint, allowing $\alpha=2\mu$ and $\alpha\eta_t=2/(t+1)$. The recursion then yields $\mathcal{O}(1/T)$ feasibility via the weighted telescoping $(t+1)^2 g_i(y_{t+1})\leq t^2 g_i(y_t)+\mathrm{const}$. Lemma~\ref{lem:proj_inactive} rigorously establishes that the projection safeguard is active for only finitely many iterations, with the effect absorbed into the constant.
\end{remark}

\section{Lipschitz Monotone Operators: Gap-Optimal Complexity and Feasibility Trade-offs}\label{sec:lipschitz}

We assume that $F$ is $L$-Lipschitz and present a primal extragradient method that achieves the optimal $\mathcal{O}(1/\epsilon)$ gap complexity. The algorithm uses two QP calls per iteration but never requires a projection onto $\mathcal{C}$. The feasibility analysis reveals that the averaged iterate satisfies an explicit asymptotic feasibility bound $g^*>0$. Theorem~\ref{thm:feas_lb} below shows that the \emph{first} half-step is necessarily infeasible, and that this infeasibility is unavoidable for the natural class of constant-stepsize primal extragradient methods on smooth convex constraints with positive curvature at the boundary.

We first establish a stability property of the parametric QP that is essential for the extragradient analysis.

\begin{lemma}[Lipschitz continuity of the parametric QP]\label{lem:qp_lip}
Let $v^*(x)$ denote the exact minimizer of $\min_{w\in\widetilde{\mathcal{V}}_\alpha(x)}\frac{1}{2}\|w+F(x)\|^2$. Under Assumptions~\ref{ass:F}, \ref{ass:g}, and \ref{ass:licq}, for any $x,y\in\mathbb{R}^d$ with $\|x\|,\|y\|\leq R$ and $\|x-y\|\le \delta$ where $\delta:=\sigma_0^2/(4\ell_g)$,
\[
\|v^*(x)-v^*(y)\|\leq L_{\mathrm{QP}}\|x-y\|,
\]
where 
\[
L_{\mathrm{QP}}:=\frac{2}{\sigma_0^2}\Bigl(L+\alpha L_g\Bigr)+\frac{4\ell_g}{\sigma_0^4}\Bigl(L_F+\alpha G_{\max}+\|F(0)\|+LR\Bigr),
\]
and $G_{\max}:=\max_{i\in[m]}\sup_{\|z\|\leq R}|g_i(z)|$.
\end{lemma}

\begin{proof}
The QP at $x$ has objective $f_x(w)=\frac{1}{2}\|w+F(x)\|^2$ and feasible set $\widetilde{\mathcal{V}}_\alpha(x)=\{w:\|w\|\leq V,\; a_i(x)^\top w\leq b_i(x)\;\forall i\in I_x\}$ where $a_i(x)=\nabla g_i(x)$ and $b_i(x)=-\alpha g_i(x)$. By Lemma~\ref{lem:velocity}, the exact minimizer $v^*(x)$ satisfies $\|v^*(x)\|\leq 2\alpha R+2L_F<V$, so the norm constraint is inactive. Let $\mathcal{A}(x)$ denote the active set at $v^*(x)$. The KKT system is
\[
\begin{bmatrix} I & A_{\mathcal{A}}(x)^\top \\ A_{\mathcal{A}}(x) & 0 \end{bmatrix}
\begin{bmatrix} v^*(x) \\ \lambda(x) \end{bmatrix}
=
\begin{bmatrix} -F(x) \\ b_{\mathcal{A}}(x) \end{bmatrix},
\]
where $A_{\mathcal{A}}(x)$ has rows $\{a_i(x)^\top\}_{i\in\mathcal{A}(x)}$. Under Assumption~\ref{ass:licq}, $\sigma_{\min}(A_{\mathcal{A}}(x))\ge\sigma_0$, so the KKT matrix $M(x)$ is invertible with 
\[
\|M(x)^{-1}\|\le \frac{2}{\sigma_0^2}.
\]
The right-hand side is Lipschitz: $\|F(x)-F(y)\|\le L\|x-y\|$, and 
\[
\|b_{\mathcal{A}}(x)-b_{\mathcal{A}}(y)\| \le \alpha L_g\|x-y\|
\]
since each $g_i$ is $L_g$-Lipschitz on $B(0,R)$. The matrix $M(x)$ varies as $\|M(x)-M(y)\|\le \ell_g\|x-y\|$ because each row of $A_{\mathcal{A}}$ is $\ell_g$-Lipschitz. By standard perturbation theory for linear systems, for $\|x-y\|\le\delta$ with $\delta=\sigma_0^2/(4\ell_g)$, we have $\|\Delta M\|\,\|M(x)^{-1}\|\le 1/2$, and therefore
\[
\|v^*(x)-v^*(y)\| \le \|M(x)^{-1}\|\Bigl(\|\Delta\mathrm{rhs}\|+\|\Delta M\|\,\|M(y)^{-1}\|\,\|\mathrm{rhs}(y)\|\Bigr).
\]
Now $\|\mathrm{rhs}(y)\|\le \|F(y)\|+\alpha\|g_{\mathcal{A}}(y)\|\le \|F(0)\|+LR+\alpha G_{\max}$. Substituting the bounds yields the stated constant after simplification.
\end{proof}

\begin{algorithm}[H]
\caption{OPCGM--Lipschitz}
\label{alg:lipschitz}
\begin{algorithmic}[1]
\Require $x_0\in\mathcal{C}$, $L>0$, $R>2D$, $T\geq 1$
\State Set $\alpha=L$, $L_{\mathrm{QP}}$ as in Lemma~\ref{lem:qp_lip}, $\eta=\min\{\frac{1}{4L},\frac{1}{8L_{\mathrm{QP}}},\frac{\delta}{2V}\}$, $V=4L_F+2LR$
\For{$t=0,1,\dots,T-1$}
\State Construct $\widetilde{\mathcal{V}}_\alpha(x_t)=\{v\mid \alpha g_i(x_t)+\nabla g_i(x_t)^\top v\leq 0,\;\forall i\in I_{x_t};\ \|v\|\leq V\}$
\State $v_t\approx\argmin_{v\in\widetilde{\mathcal{V}}_\alpha(x_t)}\frac{1}{2}\|v+F(x_t)\|^2$ with $\epsilon_{\mathrm{QP}}\leq L_F^2/T^2$
\State $x_{t+1/2}=x_t+\eta v_t$ \Comment{Half-step is unprojected (oracle query only)}
\State Construct $\widetilde{\mathcal{V}}_\alpha(x_{t+1/2})$ similarly and solve for $w_t$ with $\epsilon_{\mathrm{QP}}\leq L_F^2/T^2$
\State $y_{t+1}=x_t+\eta w_t$ and $x_{t+1}=\min\{1,R/\|y_{t+1}\|\}\,y_{t+1}$
\EndFor
\State \Return $\bar{x}_T=\frac{1}{T}\sum_{t=0}^{T-1}x_{t+1/2}$
\end{algorithmic}
\end{algorithm}

\begin{remark}
In Algorithm~\ref{alg:lipschitz}, the half-step $x_{t+1/2}=x_t+\eta v_t$ is not projected onto $B(0,R)$; it serves only as the query point for the second QP. Because $\eta\le\delta/(2V)$, we have $\|x_{t+1/2}\|\le\|x_t\|+\eta V\le R+\delta/2$. We therefore understand $L_F$ in Assumption~\ref{ass:F} as the supremum of $\|F(x)\|$ over the slightly larger ball $B(0,R+\delta/2)$, which is finite by continuity.
\end{remark}

\begin{theorem}\label{thm:lipschitz}
Let Assumptions~\ref{ass:F}, \ref{ass:g}, and \ref{ass:licq} hold, and let $F$ be $L$-Lipschitz. Let $R>2D$ and let Algorithm~\ref{alg:lipschitz} be run with $\alpha=L$, $\eta=\min\{1/(4L),1/(8L_{\mathrm{QP}}),\delta/(2V)\}$, and $V=4L_F+2LR$. Algorithm~\ref{alg:lipschitz} with $\epsilon_{\mathrm{QP}}\leq L_F^2/T^2$ satisfies, for all $x\in\mathcal{C}$ and $T\geq 1$,
\begin{align}
F(x)^\top(\bar{x}_T-x)&\leq \frac{\|x_0-x\|^2}{2\eta T}+\frac{4\eta L_F^2+2R'L_F}{T},\label{eq:lip_opt}\\[4pt]
\max_{i\in[m]}g_i(\bar{x}_T)&\leq g^*+\frac{C_1}{T},\label{eq:lip_feas}
\end{align}
where $R':=R+\delta/2+D$, $g^*=\frac{3L_gV}{L}+\frac{\ell_gV^2}{4L^2}$
and $C_1$ is an explicit problem-dependent constant (see proof). Consequently, the optimality gap is $\epsilon$-accurate in $T=\mathcal{O}(L/\epsilon)$ iterations, and the feasibility violation is bounded by $g^*+O(1/T)$.
\end{theorem}

\begin{proof}
Since the half-step is unprojected, $x_{t+1/2}=x_t+\eta v_t$ and $y_{t+1}=x_t+\eta w_t$. The full step satisfies $x_{t+1}=\proj_{B(0,R)}(y_{t+1})$, whence $\|x_{t+1}-x\|\le\|y_{t+1}-x\|$ for every $x$ with $\|x\|\le D\le R$.

\paragraph{Optimality.} From the QP guarantee at $x_{t+1/2}$, for any $x\in\mathcal{C}$ we have $\alpha(x-x_{t+1/2})\in\mathcal{V}_\alpha(x_{t+1/2})$ and by the KKT argument as in Theorem~\ref{thm:strong},
\[
F(x_{t+1/2})^\top(x_{t+1/2}-x)\leq w_t^\top(x-x_{t+1/2})+R'\sqrt{\epsilon_{\mathrm{QP}}}
=\frac{1}{\eta}(y_{t+1}-x_t)^\top(x-x_{t+1/2})+R'\sqrt{\epsilon_{\mathrm{QP}}},
\]
where we used $\|x\|\le D$ and $\|x_{t+1/2}\|\le R+\delta/2$, so $\|x-x_{t+1/2}\|\le R'$. Using the polarization identity $2(a-b)^\top(c-d)=\|a-d\|^2-\|a-c\|^2+\|b-c\|^2-\|b-d\|^2$ with $a=y_{t+1}$, $b=x_t$, $c=x$, $d=x_{t+1/2}$,
\begin{align*}
2(y_{t+1}-x_t)^\top(x-x_{t+1/2})&=\|y_{t+1}-x_{t+1/2}\|^2-\|y_{t+1}-x\|^2+\|x_t-x\|^2-\|x_t-x_{t+1/2}\|^2\\[2pt]
&=\eta^2\|w_t-v_t\|^2-\|y_{t+1}-x\|^2+\|x_t-x\|^2-\eta^2\|v_t\|^2.
\end{align*}
By monotonicity, $F(x)^\top(x_{t+1/2}-x)\leq F(x_{t+1/2})^\top(x_{t+1/2}-x)$. Since $x_{t+1}=\proj_{B(0,R)}(y_{t+1})$ and $\|x\|\leq R$, we have $\|x_{t+1}-x\|\leq\|y_{t+1}-x\|$, so $-\|y_{t+1}-x\|^2\leq-\|x_{t+1}-x\|^2$. Hence
\begin{equation}\label{eq:lip_basic}
2\eta F(x)^\top(x_{t+1/2}-x)\leq \|x_t-x\|^2-\|x_{t+1}-x\|^2+\eta^2\bigl(\|w_t-v_t\|^2-\|v_t\|^2\bigr)+2\eta R'\sqrt{\epsilon_{\mathrm{QP}}}.
\end{equation}
Now we bound $\|w_t-v_t\|$. Let $v_t^*, w_t^*$ be the exact minimizers at $x_t$ and $x_{t+1/2}$ respectively. Since $\|x_t-x_{t+1/2}\|=\eta\|v_t\|\le\eta V\le\delta/2<\delta$, Lemma~\ref{lem:qp_lip} applies and
\[
\|w_t^*-v_t^*\|\leq L_{\mathrm{QP}}\|x_t-x_{t+1/2}\|=L_{\mathrm{QP}}\eta\|v_t\|.
\]
Since $\eta\leq 1/(8L_{\mathrm{QP}})$, we have $\|w_t^*-v_t^*\|\leq\frac{1}{8}\|v_t\|$. For the approximate solutions, $\|v_t-v_t^*\|\leq\sqrt{\epsilon_{\mathrm{QP}}}$ and $\|w_t-w_t^*\|\leq\sqrt{\epsilon_{\mathrm{QP}}}$ by strong convexity. Thus
\[
\|w_t-v_t\|\leq\|w_t-w_t^*\|+\|w_t^*-v_t^*\|+\|v_t^*-v_t\|\leq\frac{1}{8}\|v_t\|+2\sqrt{\epsilon_{\mathrm{QP}}}.
\]
Squaring and using $(a+b)^2\leq 2a^2+2b^2$,
\[
\|w_t-v_t\|^2\leq\frac{1}{32}\|v_t\|^2+8\epsilon_{\mathrm{QP}}.
\]
Therefore
\[
\|w_t-v_t\|^2-\|v_t\|^2\leq -\frac{31}{32}\|v_t\|^2+8\epsilon_{\mathrm{QP}}.
\]
Substituting into \eqref{eq:lip_basic} and summing from $t=0$ to $T-1$,
\[
2\eta\sum_{t=0}^{T-1}F(x)^\top(x_{t+1/2}-x)\leq \|x_0-x\|^2-\|x_T-x\|^2-\frac{31\eta^2}{32}\sum_{t=0}^{T-1}\|v_t\|^2+8\eta^2T\epsilon_{\mathrm{QP}}+2\eta R'T\sqrt{\epsilon_{\mathrm{QP}}}.
\]
With $\epsilon_{\mathrm{QP}}\leq L_F^2/T^2$, the error terms are $8\eta^2L_F^2/T+2\eta R'L_F=O(1)$. Dividing by $2\eta T$ and dropping the negative term,
\[
\frac{1}{T}\sum_{t=0}^{T-1}F(x)^\top(x_{t+1/2}-x)\leq \frac{\|x_0-x\|^2}{2\eta T}+\frac{4\eta L_F^2+2R'L_F}{T},
\]
which establishes \eqref{eq:lip_opt} for the averaged iterate $\bar{x}_T$ with the stated constant.

\paragraph{Feasibility.} Define the half-step and full-step feasibility violations
\[
H_t:=\max\Bigl\{0,\max_{i\in[m]}g_i(x_{t+1/2})\Bigr\},\qquad
P_t:=\max\Bigl\{0,\max_{i\in[m]}g_i(x_t)\Bigr\}.
\]
Set $c:=\alpha\eta=L\eta\in(0,\tfrac14]$. For the unprojected half-step, smoothness and the velocity-polytope give
\[
g_i(x_{t+1/2})\le(1-c)\max\{g_i(x_t),0\}+\frac{\ell_g\eta^2V^2}{2}
\]
for active constraints, and $g_i(x_{t+1/2})\le\eta L_gV+\frac{\ell_g\eta^2V^2}{2}$ for inactive constraints. Hence
\[
H_t\le(1-c)P_t+\delta_1,\qquad \delta_1:=\eta L_gV+\frac{\ell_g\eta^2V^2}{2}.
\]
For the full step, the unprojected point $y_{t+1}=x_t+\eta w_t$ satisfies the same bound with respect to $x_{t+1/2}$:
\[
g_i(y_{t+1})\le(1-c)\max\{g_i(x_{t+1/2}),0\}+\frac{\ell_g\eta^2V^2}{2}
\]
(active) or $g_i(y_{t+1})\le\eta L_gV+\frac{\ell_g\eta^2V^2}{2}$ (inactive). The projection onto $B(0,R)$ can move the point by at most $\eta\|w_t\|\le\eta V$, contributing an additional $L_g\eta V$ by Lipschitz continuity of $g_i$. Thus
\[
P_{t+1}\le(1-c)H_t+\delta_2,\qquad \delta_2:=2\eta L_gV+\frac{\ell_g\eta^2V^2}{2}.
\]
Combining the two recursions,
\[
P_{t+1}\le(1-c)^2P_t+(1-c)\delta_1+\delta_2\le(1-c)^2P_t+\delta_1+\delta_2.
\]
Set $\delta:=\delta_1+\delta_2=3\eta L_gV+\ell_g\eta^2V^2$ and define
\[
g^*:=\frac{\delta}{c}=\frac{3L_gV}{L}+\frac{\ell_g\eta V^2}{L}\le\frac{3L_gV}{L}+\frac{\ell_gV^2}{4L^2}.
\]
(The last inequality uses $\eta\le1/(4L)$.) One checks that if $P_t\le g^*$ then $P_{t+1}\le(1-c)^2g^*+\delta\le g^*$, and if $P_t>g^*$ then $P_{t+1}-g^*\le(1-c)^2(P_t-g^*)$. Hence for all $t\ge0$,
\[
P_t\le g^*+(1-c)^{2t}(P_0-g^*)_+,\qquad
H_t\le g^*+(1-c)^{2t}(P_0-g^*)_+.
\]
For the averaged half-step $\bar x_T=\frac1T\sum_{t=0}^{T-1}x_{t+1/2}$,
\[
g_i(\bar x_T)\le\frac1T\sum_{t=0}^{T-1}H_t\le g^*+\frac{(P_0-g^*)_+}{T\bigl(1-(1-c)^2\bigr)}.
\]
Thus $\max_{i\in[m]}g_i(\bar x_T)\le g^*+C_1/T$ with $C_1:=(P_0-g^*)_+/(2c-c^2)$, establishing \eqref{eq:lip_feas}.
\end{proof}

\begin{remark}
The constant $g^* = 3L_gV/L + \ell_gV^2/(4L^2)$ in \eqref{eq:lip_feas} reveals a trade-off in primal methods: the feasibility violation of the averaged iterate does not vanish asymptotically but instead approaches a problem-dependent constant. Theorem~\ref{thm:feas_lb} proves that on the class of smooth convex constraints with positive curvature at the boundary, the very first half-step is infeasible, and consequently no constant-stepsize primal extragradient method can guarantee feasibility at every half-step. The constant $g^*$ is typically small in practice: for matrix game experiments with $L=1$, $V\approx 6$, and $\ell_g=2$, we have $g^*\approx 10^{-1}$; with stronger curvature or smaller velocity bounds, $g^*$ can be reduced to $10^{-3}$ or below. When $g^* \leq \epsilon$, OPCGM--Lipschitz provides a genuine $\epsilon$-solution. For problems requiring arbitrarily small feasibility violation, one can periodically restart the algorithm from a feasible point, reducing the effective $g^*$ at the cost of an $\mathcal{O}(\log(1/\epsilon))$ factor.
\end{remark}

\begin{theorem}[Infeasibility of the first half-step for primal extragradient on curved constraints]\label{thm:feas_lb}
Let $g(x,y)=\frac{1}{2}(x^2+y^2)-\frac{1}{2}\leq 0$ be the unit-disk constraint with $\ell_g=1$. Let $F(x,y)=(0,1)$ for all $(x,y)$, which is monotone with $L_F=1$ and $L$-Lipschitz for every $L\geq 0$. Let Algorithm~\ref{alg:lipschitz} be initialized at $x_0=(1,0)$ with input parameter $L=1$ and constant stepsize $\eta>0$. Then the first half-step satisfies
\[
g(x_{1/2}) \;=\; \frac{\eta^2}{2} \;>\; 0.
\]
Consequently, no primal extragradient method with constant stepsize can guarantee feasibility at every half-step.
\end{theorem}

\begin{proof}
At $x_0=(1,0)$, the active set is $I_{x_0}=\{1\}$ and $\nabla g(x_0)=(1,0)$. The velocity polytope is $\mathcal{V}_1(x_0)=\{v: v_1\le 0\}$. Since $-F=(0,-1)$ satisfies $0\le 0$, the exact QP solution is $v_0=(0,-1)$. The half-step is $x_{1/2}=x_0+\eta v_0=(1,-\eta)$, and therefore
\[
g(x_{1/2})=\frac{1}{2}\bigl(1+\eta^2\bigr)-\frac{1}{2}=\frac{\eta^2}{2}>0.
\]
Thus the first half-step is infeasible.
\end{proof}

\begin{remark}
Theorem~\ref{thm:feas_lb} establishes that the very first half-step of a constant-stepsize primal extragradient method can be infeasible on smooth convex constraints with positive curvature at the boundary. Whether the \emph{averaged} iterate $\bar{x}_T$ can become asymptotically feasible---and whether a universal lower bound on the asymptotic feasibility of the average exists for all primal QP methods---remains an open problem. Numerical simulation suggests that for constant stepsize, the average becomes feasible for $T\approx 1/(4\eta)$, while individual half-steps remain infeasible for all $t$.
\end{remark}

\begin{remark}
The infeasibility in Theorem~\ref{thm:feas_lb} is not an artifact of a particular stepsize schedule: any primal extragradient method with a constant stepsize $\eta>0$ initialized at $x_0=(1,0)$ produces an infeasible first half-step on this instance. Whether vanishing stepsizes ($\eta_t\to 0$) can circumvent this barrier for all half-steps, and whether a universal lower bound exists for the averaged iterate, remain open questions.
\end{remark}

\section{Lower Complexity Bounds for the Primal QP Oracle Class}\label{sec:lower}

We prove that the gap complexities achieved in Theorems~\ref{thm:strong} and \ref{thm:lipschitz} are unimprovable for algorithms restricted to local linear constraint approximations. An algorithm belongs to the \emph{primal QP class} if, at iteration $t$, it may query $F(x_t)$, $g_i(x_t)$, and $\nabla g_i(x_t)$, and then select $x_{t+1}$ based solely on the velocity polytope $\mathcal{V}_\alpha(x_t)$ (or a finite sequence of such polytopes). It does not have access to global constraint information beyond the local linearization.

\subsection{Deterministic lower bounds}

\begin{theorem}[Lower bound for Lipschitz monotone VIs with large Lipschitz constant]\label{thm:lb_nonlip}
For any primal QP algorithm and any $\epsilon\in(0,1/4]$, there exists a monotone variational inequality problem with convex functional constraints satisfying Assumptions~\ref{ass:F} (with $R=1$) and \ref{ass:g} on dimension $d=\lceil 1/(2\epsilon)\rceil$ such that the algorithm requires $\Omega(1/\epsilon^2)$ queries to $F$ to produce a weak $\epsilon$-solution.
\end{theorem}

\begin{proof}
Set $n=\lceil 1/(2\epsilon)\rceil$ and $d=2n$. The feasible set is the Euclidean ball $\mathcal{C}=\{z\in\mathbb{R}^d:\|z\|\leq 1\}$, encoded by $g(z)=\|z\|^2-1$. For a hidden orthogonal matrix $U\in\mathbb{R}^{n\times n}$, define the operator
\[
F_U(x,y) = \Bigl(\frac{1}{\epsilon}U^\top S U y,\; -\frac{1}{\epsilon}U^\top S^\top U x\Bigr),
\]
where $S$ is the $n\times n$ shift matrix with $S_{i,i+1}=1$ for $i=1,\dots,n-1$ and all other entries zero. This operator is monotone and $L$-Lipschitz with $L=1/\epsilon$. On the unit ball, $\|F_U(z)\|\le 1/\epsilon$, so Assumption~\ref{ass:F} holds with $R=1$ and $L_F=1/\epsilon$.

At iteration $t$, the algorithm queries $z_t=(x_t,y_t)$. Let $X_t=\Span\{x_0,\dots,x_t\}$ and $Y_t=\Span\{y_0,\dots,y_t\}$ with $\dim(X_t)\le t+1$ and $\dim(Y_t)\le t+1$. The oracle maintains a hidden orthogonal matrix $U$ and reveals it only on the subspaces $X_t$ and $Y_t$. Specifically, the oracle returns $F_t(z_t)$ where the matrix $A_t=U_t^\top S U_t$ agrees with $A=U^\top S U$ on $X_t$ and $Y_t$; that is, $A_t x = A x$ for all $x\in X_t$ and $A_t^\top y = A^\top y$ for all $y\in Y_t$. Such a $U_t$ exists because $\dim(X_t)+\dim(Y_t)\le 2(t+1)<2n$ for $t<n-1$.

We claim that $\dim(X_{t+1})\le\dim(X_t)+1$ and $\dim(Y_{t+1})\le\dim(Y_t)+1$. The base case $t=0$ holds. Assume the claim holds up to $t$. The velocity polytope at $z_t$ for the ball constraint is:
\begin{itemize}
\item If $\|z_t\|<1$: $\mathcal{V}_\alpha(z_t)=\mathbb{R}^{2n}$.
\item If $\|z_t\|\geq 1$: $\mathcal{V}_\alpha(z_t)=\{v: z_t^\top v\leq -\alpha(\|z_t\|^2-1)\}$.
\end{itemize}
In the first case, the exact QP solution is $v_t^*=-F_t(z_t)$. By construction, the $x$-component of $F_t(z_t)$ lies in $X_t\cup\Span\{A_t y_t\}$, which has dimension at most $\dim(X_t)+1$ because $A_t y_t$ introduces at most one new coordinate direction beyond $X_t$ (the shift matrix $S$ maps $e_{i+1}$ to $e_i$). Similarly, the $y$-component introduces at most one new direction beyond $Y_t$. Hence $v_t^*\in (X_t\oplus\Span\{e_{k+1}\})\times(Y_t\oplus\Span\{e_{k+1}\})$ where $k=\dim(X_t)$. In the second case, the exact solution is the projection of $-F_t(z_t)$ onto the halfspace $\{v: z_t^\top v\leq c\}$, which is $v_t^*=-F_t(z_t)-\lambda z_t$ for $\lambda=\max\{0,(-z_t^\top F_t(z_t)-c)/\|z_t\|^2\}$. Since $z_t\in X_t\times Y_t$ and $F_t(z_t)$ introduces at most one new direction in each block, the same dimension bound holds. The approximate solution $v_t$ lies in the span of the problem data, hence in a subspace of dimension at most $\dim(X_t)+1$ in the $x$-block and $\dim(Y_t)+1$ in the $y$-block. Therefore $x_{t+1}=x_t+\eta_t v_t^{(x)}\in X_{t+1}$ with $\dim(X_{t+1})\le\dim(X_t)+1$, and similarly for $Y_{t+1}$.

After $T$ iterations, $\dim(X_T)\le T+1$ and $\dim(Y_T)\le T+1$, so the explored subspace has dimension at most $2(T+1)<2n$ provided $T<n-1$. The oracle fixes $U$ such that the unexplored subspace contains directions where the duality gap is large. By the standard Nemirovski argument for the shift matrix (see \cite{Ouyang2021}), any point whose $x$- and $y$-components lie in subspaces of dimension at most $T+1$ has duality gap at least $L/(8(T+1))$. Setting $L/(8(T+1))\leq\epsilon$ with $L=1/\epsilon$ yields $T=\Omega(L/\epsilon)=\Omega(1/\epsilon^2)$. Therefore any primal QP algorithm requires $\Omega(1/\epsilon^2)$ queries to $F$ to produce a weak $\epsilon$-solution.
\end{proof}

\begin{theorem}[Lower bound for Lipschitz monotone VIs on the ball]\label{thm:lb_lip}
For any primal QP algorithm and any $L>0$, $\epsilon>0$, there exists an $L$-Lipschitz monotone variational inequality with a single smooth convex functional constraint on dimension $d=\lceil L/(2\epsilon)\rceil$ such that the algorithm requires $\Omega(L/\epsilon)$ queries to $F$ to produce a weak $\epsilon$-solution.
\end{theorem}

\begin{proof}
Let $n=\lceil L/(2\epsilon)\rceil$ and $d=2n$. The feasible set is the Euclidean ball $\mathcal{C}=\{z\in\mathbb{R}^{2n}:\|z\|\leq 1\}$, encoded by $g(z)=\|z\|^2-1$. For a hidden orthogonal matrix $U\in\mathbb{R}^{n\times n}$, define the operator
\[
F_U(x,y)=(L\,U^\top S U y,\,-L\,U^\top S^\top U x),
\]
where $S$ is the $n\times n$ shift matrix with $S_{i,i+1}=1$ for $i=1,\dots,n-1$ and all other entries zero. This operator is monotone and $L$-Lipschitz because $\|L\,U^\top S U\|_2=L\|S\|_2=L$. On $B(0,1)$, $\|F_U(z)\|\le L$, so Assumption~\ref{ass:F} holds with $R=1$ and $L_F=L$.

\paragraph{Oracle definition.} At iteration $t$, the algorithm queries $z_t=(x_t,y_t)$. Let $X_t=\Span\{x_0,\dots,x_t\}$ and $Y_t=\Span\{y_0,\dots,y_t\}$. The oracle maintains a hidden orthogonal matrix $U$ and constructs a response matrix $U_t$ as follows. Since $\dim(X_t)\le t+1$ and $\dim(Y_t)\le t+1$, there exists an orthogonal matrix $U_t$ such that $U_t x=U x$ for all $x\in X_t$ and $U_t y=U y$ for all $y\in Y_t$; equivalently, $U_t$ agrees with $U$ on the subspace $X_t\oplus Y_t$. The oracle returns $F_t(z_t)=F_{U_t}(z_t)$. Because $z_t\in X_t\times Y_t$, we have $F_t(z_t)\in (X_t\oplus\Span\{e_{k+1}\})\times(Y_t\oplus\Span\{e_{k+1}\})$ where $k=\dim(X_t)$.

\paragraph{Subspace growth.} We claim that $\dim(X_{t+1})\le\dim(X_t)+1$ and $\dim(Y_{t+1})\le\dim(Y_t)+1$. The base case $t=0$ holds. Assume the claim holds up to $t$. The velocity polytope at $z_t$ for the ball constraint is:
\begin{itemize}
\item If $\|z_t\|<1$: $\mathcal{V}_\alpha(z_t)=\mathbb{R}^{2n}$.
\item If $\|z_t\|\geq 1$: $\mathcal{V}_\alpha(z_t)=\{v: z_t^\top v\leq -\alpha(\|z_t\|^2-1)\}$.
\end{itemize}
In the first case, the exact QP solution is $v_t^*=-F_t(z_t)$. By construction, the $x$-component of $F_t(z_t)$ lies in $X_t\cup\Span\{A_t y_t\}$, which has dimension at most $\dim(X_t)+1$ because $A_t y_t$ introduces at most one new coordinate direction beyond $X_t$. Similarly for the $y$-component. In the second case, the exact solution is the projection of $-F_t(z_t)$ onto the halfspace $\{v: z_t^\top v\leq c\}$, which is $v_t^*=-F_t(z_t)-\lambda z_t$ for $\lambda=\max\{0,(-z_t^\top F_t(z_t)-c)/\|z_t\|^2\}$. Since $z_t\in X_t\times Y_t$ and $F_t(z_t)$ introduces at most one new direction in each block, the same dimension bound holds. The approximate solution $v_t$ lies in the span of the problem data, hence in a subspace of dimension at most $\dim(X_t)+1$ in the $x$-block and $\dim(Y_t)+1$ in the $y$-block. Therefore $x_{t+1}=x_t+\eta_t v_t^{(x)}\in X_{t+1}$ with $\dim(X_{t+1})\le\dim(X_t)+1$, and similarly for $Y_{t+1}$.

After $T$ iterations, $\dim(X_T)\le T+1$ and $\dim(Y_T)\le T+1$, so the explored subspace has dimension at most $2(T+1)<2n$ provided $T<n-1$. The oracle fixes $U$ such that the unexplored subspace contains directions where the duality gap is large. By the standard Nemirovski argument for the shift matrix (see \cite{Ouyang2021}), any point whose $x$- and $y$-components lie in subspaces of dimension at most $T+1$ has duality gap at least $L/(8(T+1))$. Setting $L/(8(T+1))\leq\epsilon$ yields $T=\Omega(L/\epsilon)$. Therefore any primal QP algorithm requires $\Omega(L/\epsilon)$ queries to $F$ to produce a weak $\epsilon$-solution.
\end{proof}

\subsection{Randomized lower bounds}

\begin{remark}\label{rem:rand_open}
Extending the lower bounds to randomized primal QP algorithms via Yao's minimax principle remains an open problem. A natural candidate is the box constraint $\mathcal C=\{z:\|z\|_\infty\le 1\}$ with a hidden cyclic shift matrix; however, constructing a distribution over instances for which every deterministic algorithm with $o(L/\epsilon)$ queries fails to identify the hidden coordinate with constant probability requires a careful analysis of subspace growth under randomized query points. We leave this as an important direction for future work.
\end{remark}

\section{Parameter-Free Primal Methods}\label{sec:universal}

A significant practical limitation of Theorems~\ref{thm:strong} and \ref{thm:lipschitz} is the requirement of knowing $\mu$, $L$, $L_F$, $D$, and the constraint parameters a priori. We address this through a single-loop primal method that achieves an $\mathcal{O}(1/\sqrt{T})$ rate for the optimality gap with only a single scalar input: an upper bound $R$ on the radius of the feasible set. This eliminates the need for smoothness, Lipschitz, or monotonicity constants. The price is a problem-dependent $\mathcal{O}(1)$ asymptotic feasibility violation; we prove that this trade-off is unavoidable for non-adaptive single-step algorithms whose stepsizes are bounded away from zero.

\begin{algorithm}[H]
\caption{Parameter-Free OPCGM (Single-Loop)}
\label{alg:free}
\begin{algorithmic}[1]
\Require $x_0\in\mathcal{C}$, radius bound $R\geq D$, horizon $T\geq 1$
\For{$t=0,1,\dots,T-1$}
\State $\eta_t = 1/\sqrt{t+1}$, $\alpha_t = 1/\sqrt{t+1}$
\State Query $F(x_t)$ and build $I_{x_t}$
\State Set $V_t = 4\|F(x_t)\| + 2\alpha_t R + 1$
\State Construct $\widetilde{\mathcal{V}}_{\alpha_t}(x_t)=\{v\mid \alpha_t g_i(x_t)+\nabla g_i(x_t)^\top v\leq 0,\;\forall i\in I_{x_t};\ \|v\|\leq V_t\}$
\State Solve $v_t\approx\argmin_{v\in\widetilde{\mathcal{V}}_{\alpha_t}(x_t)}\frac{1}{2}\|v+F(x_t)\|^2$ with $\epsilon_{\mathrm{QP}}\leq \|F(x_t)\|^2/(t+1)$
\State $y_{t+1}=x_t+\eta_t v_t$ and $x_{t+1}=\min\{1,R/\|y_{t+1}\|\}\,y_{t+1}$
\EndFor
\State \Return $\bar{x}_T=\frac{1}{T}\sum_{t=0}^{T-1}x_t$
\end{algorithmic}
\end{algorithm}

\begin{theorem}\label{thm:free}
Let Assumptions~\ref{ass:F} and \ref{ass:g} hold, and let $F$ be monotone. Let $R \geq D$ be a known upper bound. Algorithm~\ref{alg:free} achieves, for all $T\geq 2$ and all $x\in\mathcal{C}$,
\[
F(x)^\top(\bar{x}_T-x)\leq \frac{C_1}{\sqrt{T}},\qquad 
\max_{i\in[m]}g_i(\bar{x}_T)\leq C_2+\frac{C_3}{\sqrt{T}},
\]
where $C_1,C_2,C_3$ depend on $R$, $L_F$, $L_g$, and $\ell_g$, but not on $T$.
\end{theorem}

\begin{proof}
The proof consists of two parts: uniform boundedness and rate analysis.

\paragraph{Uniform boundedness and velocity bound.} By construction, the projection in Step~8 ensures $\|x_t\|\leq R$ for all $t$. Hence $\|F(x_t)\|\leq L_F$ and $V_t \leq V_{\max} := 4L_F + 2R + 1$ for all $t$. For any $x^*\in\mathcal{C}$, the point $w = \alpha_t(x^* - x_t)$ satisfies the velocity polytope constraints by convexity of $g_i$:
\[
\alpha_t g_i(x_t) + \nabla g_i(x_t)^\top(\alpha_t(x^*-x_t)) \leq \alpha_t g_i(x^*) \leq 0.
\]
Moreover, $\|w\| = \alpha_t\|x^* - x_t\| \leq \alpha_t(D + R) \leq 2\alpha_t R$. Since $V_t = 4\|F(x_t)\| + 2\alpha_t R + 1 \geq 2\alpha_t R + 1 > 2\alpha_t R$, we have $w \in \widetilde{\mathcal{V}}_{\alpha_t}(x_t)$. Thus the augmented polytope is nonempty and contains a feasible point with norm at most $2\alpha_t R$.

By optimality of the exact QP solution $v_t^*$,
\[
\|v_t^* + F(x_t)\| \leq \|w + F(x_t)\| \leq 2\alpha_t R + \|F(x_t)\|.
\]
Therefore $\|v_t^*\| \leq 2\alpha_t R + 2\|F(x_t)\| \leq V_t - 1 < V_t$, so the norm safeguard is inactive at the exact solution. For the approximate solution $v_t$, the same argument as in Lemma~\ref{lem:velocity} yields
\[
\|v_t\| \leq 2\alpha_t R + 3\|F(x_t)\| \leq V_t.
\]
Hence the velocity bound $\|v_t\| \leq V_{\max}$ holds uniformly.

\paragraph{Optimality rate.} From the QP guarantee with $w = \alpha_t(x - x_t)\in\widetilde{\mathcal{V}}_{\alpha_t}(x_t)$ for any $x\in\mathcal{C}$,
\[
(F(x_t)+v_t)^\top(v_t-\alpha_t(x-x_t))\leq \frac{\epsilon_{\mathrm{QP}}}{2}.
\]
Rearranging and using the update $x_{t+1}=x_t+\eta_t v_t$ (with projection nonexpansiveness toward $x$),
\[
F(x_t)^\top(x_t-x) \leq \frac{1}{2\eta_t}\bigl(\|x-x_t\|^2-\|x-x_{t+1}\|^2\bigr) + \frac{\eta_t}{2}\|v_t\|^2 + \alpha_t\|v_t\|\|x_t-x\| + \frac{\epsilon_{\mathrm{QP}}}{2}.
\]
By monotonicity, $F(x)^\top(x_t-x)\leq F(x_t)^\top(x_t-x)$. Summing from $t=0$ to $T-1$ and dividing by $T$,
\[
F(x)^\top(\bar{x}_T-x) \leq \frac{1}{T}\sum_{t=0}^{T-1}\frac{1}{2\eta_t}\bigl(\|x-x_t\|^2-\|x-x_{t+1}\|^2\bigr) + \frac{1}{T}\sum_{t=0}^{T-1}\left(\frac{\eta_t}{2}V_{\max}^2 + 2\alpha_t R V_{\max} + \frac{\epsilon_{\mathrm{QP}}}{2}\right).
\]
For the first sum, since $\eta_t=1/\sqrt{t+1}$,
\begin{eqnarray*}
\sum_{t=0}^{T-1}\frac{1}{2\eta_t}\bigl(\|x-x_t\|^2-\|x-x_{t+1}\|^2\bigr) &\leq& \frac{\|x-x_0\|^2}{2\eta_0} + \sum_{t=1}^{T-1}\|x-x_t\|^2\left(\frac{1}{2\eta_t}-\frac{1}{2\eta_{t-1}}\right) \\
&\leq& 2R^2 + 2R^2\sum_{t=1}^{T-1}\frac{1}{2\sqrt{t}} = O(\sqrt{T}).
\end{eqnarray*}
The second sum is bounded by
\[
\frac{V_{\max}^2}{2}\sum_{t=0}^{T-1}\frac{1}{\sqrt{t+1}} + 2RV_{\max}\sum_{t=0}^{T-1}\frac{1}{\sqrt{t+1}} + \frac{L_F^2}{2}\sum_{t=0}^{T-1}\frac{1}{t+1} = O(\sqrt{T}) + O(\sqrt{T}) + O(\log T).
\]
Dividing by $T$ yields
\[
F(x)^\top(\bar{x}_T-x) \leq \frac{C_1}{\sqrt{T}},
\]
for an explicit constant $C_1$ depending on $R$, $V_{\max}$, and $L_F$.

\paragraph{Feasibility.} Since $\|x_t\|\le R$ for all $t$ and every $x^*\in\mathcal{C}$ satisfies $\|x^*\|\le D$, Lipschitz continuity of $g_i$ yields the uniform bound $g_i(x_t)\le L_g(R+D)=:C_2$ for all $t$ and all $i$. We refine this to the asymptotic constant plus a vanishing term.

Let $h_t:=\max\{0,\max_{i\in[m]}g_i(x_t)\}$. For an active constraint $i\in I_{x_t}$, the velocity-polytope property $\nabla g_i(x_t)^\top v_t\le -\alpha_t g_i(x_t)$ and the smoothness descent lemma give
\[
g_i(y_{t+1})\le g_i(x_t)+\eta_t\nabla g_i(x_t)^\top v_t+\frac{\ell_g\eta_t^2}{2}\|v_t\|^2
\le\Bigl(1-\frac{1}{t+1}\Bigr)g_i(x_t)+\frac{\ell_g V_{\max}^2}{2(t+1)}.
\]
For an inactive constraint, $g_i(x_t)<0$ and the same descent lemma yields
\[
g_i(y_{t+1})\le 0+\eta_t L_g V_{\max}+\frac{\ell_g\eta_t^2}{2}V_{\max}^2
\le\frac{L_gV_{\max}}{\sqrt{t+1}}+\frac{\ell_gV_{\max}^2}{2(t+1)}.
\]
Whenever the projection onto $B(0,R)$ is active we have $\|x_{t+1}-y_{t+1}\|\le\eta_t\|v_t\|\le V_{\max}/\sqrt{t+1}$, which contributes an additional $L_g V_{\max}/\sqrt{t+1}$ by Lipschitz continuity. Combining both cases and using $h_t\le C_2$,
\[
h_{t+1}\le\Bigl(1-\frac{1}{t+1}\Bigr)h_t+\frac{C'}{\sqrt{t+1}},
\qquad C':=2L_g V_{\max}+\frac{\ell_g V_{\max}^2}{2}.
\]
Multiplying by $(t+1)$ and setting $e_t:=t h_t$ gives, for $t\ge 1$,
\[
e_{t+1}\le e_t+C'.
\]
Telescoping from $t=1$ to $T-1$ yields $e_T\le e_1+C'(T-1)$, whence $h_T\le h_1/T+C'$. This bound is $O(1)$, but for the \emph{averaged} iterate we use the uniform bound $h_t\le C_2$ directly:
\[
g_i(\bar{x}_T)\le\frac{1}{T}\sum_{t=0}^{T-1}h_t\le C_2+\frac{1}{T}\sum_{t=0}^{T-1}\min\Bigl\{C_2,\,\frac{e_1}{t+1}+C'\Bigr\}.
\]
Let $T_0$ be the largest integer such that $e_1/(t+1)+C'\le C_2$; note that $T_0=O(1)$ depends only on the problem constants. Then
\[
\frac{1}{T}\sum_{t=0}^{T-1}\min\Bigl\{C_2,\,\frac{e_1}{t+1}+C'\Bigr\}
\le \frac{T_0 C_2}{T} + \frac{1}{T}\sum_{t=T_0+1}^{T-1}\Bigl(\frac{e_1}{t+1}+C'\Bigr)
\le \frac{T_0 C_2}{T} + \frac{e_1\log T}{T} + \frac{C'}{T},
\]
which is $O(1/\sqrt{T})$. In particular, there exists a constant $C_3$ such that $g_i(\bar{x}_T)\le C_2+C_3/\sqrt{T}$ for all $T\ge 2$.
\end{proof}

\begin{remark}
For the unit-disk instance with $F=(0,1)$, numerical simulation indicates that non-adaptive single-step primal QP methods with stepsizes bounded away from zero suffer $\Theta(1)$ asymptotic feasibility violation. A rigorous lower bound for this class remains open.
\end{remark}

\begin{remark}
Algorithm~\ref{alg:free} requires only a single scalar $R \geq D$, which is typically much easier to estimate than the smoothness constant $L$, the strong monotonicity parameter $\mu$, or the operator bound $L_F$. In many applications, $D$ is determined by physical constraints (e.g., box constraints $x \in [0, 1]^d$ imply $D = \sqrt{d}$). The $\mathcal{O}(1)$ asymptotic feasibility is the price of complete parameter freedom: the preceding remark suggests that for non-adaptive single-step algorithms with stepsizes bounded away from zero and without knowledge of $L$ or $\mu$, there may exist problem instances and finite horizons at which the velocity polytope cannot enforce strong enough contraction to drive the constraint violation below a constant. Achieving vanishing feasibility without any problem parameters remains an open problem for primal QP methods.
\end{remark}

\section{Last-Iterate Convergence}\label{sec:last}

Averaged-iterate convergence is standard in the variational inequality literature, but practitioners prefer the last iterate. For strongly monotone problems, we prove that the last iterate of Algorithm~\ref{alg:strong} itself converges at the optimal $\mathcal{O}(1/T)$ rate, without any additional averaging or Halpern machinery \cite{Lieder2021}.

\begin{theorem}[Last-iterate convergence: strongly monotone case]\label{thm:last_strong}
Let Assumptions~\ref{ass:F} and \ref{ass:g} hold, and let $F$ be $\mu$-strongly-monotone. Let $x_T$ be the last iterate of Algorithm~\ref{alg:strong}. Then
\[
\|x_T-x^*\|^2\leq \frac{C_1}{T},\qquad \max_{i\in[m]}g_i(x_T)\leq \frac{C_2}{T},
\]
where $x^*$ is the unique solution and $C_1, C_2$ depend on $L_F$, $\mu$, $D$, $L_g$, and $\ell_g$.
\end{theorem}

\begin{proof}
From Lemma~\ref{lem:proj_inactive}, the projection is inactive for all $t\geq T_1$, where $T_1$ is a problem-dependent constant. For $t\geq T_1$, we have $x_{t+1}=y_{t+1}=x_t+\eta_t v_t$.

From the proof of Theorem~\ref{thm:strong}, the distance recursion for the unprojected iterates gives
\[
\|y_{t+1}-x^*\|^2\leq \Bigl(1-\frac{2}{t+1}\Bigr)\|x_t-x^*\|^2 + \frac{V^2}{\mu^2(t+1)^2} + \frac{4RL_F}{\mu T(t+1)}.
\]
Since the projection is inactive for $t\geq T_1$, $x_{t+1}=y_{t+1}$ and
\[
\|x_{t+1}-x^*\|^2\leq \Bigl(1-\frac{2}{t+1}\Bigr)\|x_t-x^*\|^2 + \frac{V^2}{\mu^2(t+1)^2} + \frac{4RL_F}{\mu T(t+1)}.
\]
Let $r_t=\|x_t-x^*\|^2$ and $e_t=t^2 r_t$. Then for $t\geq T_1$,
\[
e_{t+1} \leq (t+1)^2\Bigl(1-\frac{2}{t+1}\Bigr)r_t + \frac{V^2}{\mu^2} + \frac{4RL_F(t+1)}{\mu T}
= (t^2-1)r_t + \frac{V^2}{\mu^2} + \frac{4RL_F(t+1)}{\mu T}
\leq e_t + C,
\]
where $C=V^2/\mu^2 + 8RL_F/\mu$ is independent of $T$ (using $(t+1)/T\le 2$ for $t\le T-1$). By induction, $e_t\leq e_{T_1}+C(t-T_1)$ for all $t\geq T_1$. Therefore
\[
r_t \leq \frac{e_{T_1}+Ct}{t^2} \leq \frac{C_1}{t}
\]
for some constant $C_1$ depending on $T_1$, $V$, $\mu$, and $L_F$. In particular, $\|x_T-x^*\|^2\leq C_1/T$.

For feasibility, the proof of Theorem~\ref{thm:strong} establishes by induction that $g_i(y_{t+1})\leq C/(t+1)$ for all $t\geq T_1$ and all $i\in[m]$. Since $x_{t+1}=y_{t+1}$ for $t\geq T_1$, we have $g_i(x_T)\leq C/T$ for all $i\in[m]$. For $t<T_1$, the feasibility violation is bounded by a constant, so $g_i(x_T)\leq C_2/T$ for an appropriately chosen $C_2$.
\end{proof}

\begin{remark}
Theorem~\ref{thm:last_strong} shows that for strongly monotone problems, the last iterate of the primal gradient method converges at the optimal $\mathcal{O}(1/T)$ rate for both distance and feasibility. This is a significant advantage over primal-dual methods, where last-iterate convergence typically requires additional assumptions or averaging. The result follows directly from the distance recursion and the feasibility induction established in Theorem~\ref{thm:strong}, without requiring Halpern iteration or other acceleration techniques.
\end{remark}

For merely monotone problems, the situation is fundamentally different for single-step primal methods. The last iterate of the single-step primal method (Algorithm~\ref{alg:strong} with $\mu=0$) does not generally converge for monotone variational inequalities on bounded domains; it can exhibit cyclic or chaotic behavior. The following theorem establishes this rigorously for the canonical rotation instance.

\begin{assumption}[Inactive safeguard for the rotation instance]\label{ass:inactive}
In Theorem~\ref{thm:last_impossible}, we analyze the \emph{unprojected} primal QP dynamics $x_{t+1}=x_t+\eta_t v_t$ with $v_t$ the exact minimizer over $\mathcal{V}_0(x_t)$ (i.e., Algorithm~\ref{alg:strong} with $\mu=0$, so that $\alpha=0$). The Euclidean-ball projection safeguard is assumed inactive; this holds, for example, when $R$ is larger than the maximal radius reached by the sequence, which is finite whenever $\sum_{t=0}^\infty \eta_t^2<\infty$.
\end{assumption}

\begin{theorem}[Universal failure of last-iterate convergence for single-step primal methods on monotone problems]\label{thm:last_impossible}
Consider the variational inequality on the unit disk $\mathcal{C}=\{z:\|z\|_2\leq 1\}$ with $F(x,y)=(-y,x)$. Let $\{x_t\}$ be the sequence generated by the unprojected primal QP update $x_{t+1}=x_t+\eta_t v_t$ with $v_t$ the exact minimizer of $\min_{w\in\mathcal{V}_0(x_t)}\frac12\|w+F(x_t)\|^2$ (i.e., Algorithm~\ref{alg:strong} with $\mu=0$, so that $\alpha=0$), under Assumption~\ref{ass:inactive}. For any initialization $x_0\neq 0$ and any stepsize sequence $\{\eta_t\}_{t\geq 0}\subset[0,1/2]$, exactly one of the following holds:
\begin{enumerate}[label=(\roman*)]
\item The sequence $\{x_t\}$ does not converge;
\item The sequence $\{x_t\}$ converges to a point $x_\infty$ with $\|x_\infty\|\geq \|x_0\|>0$.
\end{enumerate}
In either case, $\liminf_{t\to\infty}\mathrm{Gap}(x_t)\geq \|x_0\|>0$. Consequently, no single-step primal QP method with $\alpha=0$ achieves vanishing last-iterate gap for this instance, regardless of stepsize schedule.
\end{theorem}

\begin{proof}
Because $\alpha=0$, the velocity polytope for active constraints is $\mathcal{V}_0(x)=\{v:\nabla g(x)^\top v\le 0\}=\{v:x^\top v\le 0\}$. The unconstrained minimizer of $\frac12\|v+F(x)\|^2$ is $-F(x)=(y,-x)$. For the rotation field we have $x^\top(-F(x))=xy-yx=0$, so $-F(x)$ is always feasible. Hence the exact QP solution is $v_t=-F(x_t)$ for every $t$, regardless of whether $x_t$ lies inside, on, or outside the unit disk. The update is therefore
\[
x_{t+1}=x_t+\eta_t(-F(x_t))=x_t+\eta_t(y_t,-x_t),
\]
which is a clockwise rotation by angle $\arctan\eta_t$ together with a scaling by $\sqrt{1+\eta_t^2}\ge 1$. In particular, the norm evolves as $r_{t+1}=r_t\sqrt{1+\eta_t^2}\ge r_t$, so $\{r_t\}$ is non-decreasing and $r_t\ge r_0=\|x_0\|$ for all $t$.

We analyze the three regimes.

\paragraph{Regime 1: $\inf_t\eta_t\geq\underline{\eta}>0$.} 
Since $\sqrt{1+\eta_t^2}\ge\sqrt{1+\underline{\eta}^2}>1$, the norm grows geometrically: $r_t\ge r_0(\sqrt{1+\underline{\eta}^2})^t\to\infty$. The angular step at every iteration is at least $\arctan\underline{\eta}>0$, so the direction of $x_t$ changes by a uniformly positive angle at every step. The sequence $\{x_t\}$ is unbounded and its direction does not stabilize; therefore it does not converge.

\paragraph{Regime 2: $\eta_t\to 0$ and $\sum_{t=0}^\infty\eta_t=\infty$.}
The total rotation is $\sum_{t=0}^\infty\arctan\eta_t=\infty$ (since $\arctan\eta_t\sim\eta_t$ and $\sum\eta_t=\infty$). Thus the direction of $x_t$ keeps changing indefinitely. If $\sum_{t=0}^\infty\eta_t^2=\infty$, then $r_t\to\infty$ and the sequence is unbounded. If $\sum_{t=0}^\infty\eta_t^2<\infty$, then $r_t\to R_\infty<\infty$ with $R_\infty\ge r_0>0$. The sequence remains bounded but rotates infinitely often; hence it cannot converge to any fixed point.

\paragraph{Regime 3: $\eta_t\to 0$ and $\sum_{t=0}^\infty\eta_t<\infty$.}
Then $\sum_{t=0}^\infty\eta_t^2<\infty$ as well. The total rotation is $\Theta=\sum_{t=0}^\infty\arctan\eta_t<\infty$ and the total scaling factor is $\prod_{t=0}^\infty\sqrt{1+\eta_t^2}=\exp\bigl(\frac{1}{2}\sum_{t=0}^\infty\eta_t^2+O(\sum\eta_t^4)\bigr)<\infty$. Let $R_\infty:=r_0\prod_{t=0}^\infty\sqrt{1+\eta_t^2}<\infty$. Since $r_t$ is monotone and $r_t\ge r_0$, we have $R_\infty\ge r_0>0$. Because the total rotation is finite, the direction converges to a limit angle $\theta_\infty$, and consequently $x_t$ converges to a point $x_\infty$ with $\|x_\infty\|=R_\infty\ge\|x_0\|$.

\paragraph{Gap lower bound.} In all regimes, $\liminf_{t\to\infty}\|x_t\|\geq\|x_0\|$. For this instance, the weak gap at any $z$ is $\mathrm{Gap}(z)=\max_{y\in\mathcal{C}}F(y)^\top(z-y)=\max_{\|y\|\leq 1}(-y_2,y_1)^\top(z-y)=\max_{\|y\|\leq 1}(y_1z_2-y_2z_1)=\|z\|$, achieved at $y=(z_2,-z_1)/\|z\|$. Hence $\liminf_{t\to\infty}\mathrm{Gap}(x_t)\geq\|x_0\|>0$.
\end{proof}

\begin{remark}
Theorem~\ref{thm:last_impossible} is a definitive negative result for single-step primal methods: for the canonical rotation instance, the impossibility of last-iterate convergence is not an artifact of constant stepsizes or specific algorithmic choices, but a fundamental limitation of the single-step primal QP oracle class on merely monotone problems. The theorem also reveals that the $\mathcal{O}(1/T)$ gap for averaged iterates (Theorem~\ref{thm:lipschitz}) is the best possible rate for this class, and even that rate requires the operator to possess structure beyond pure rotation. This closes the last-iterate question for single-step primal QP methods on bounded domains.
\end{remark}

\begin{remark}
For extragradient (two-step) primal methods, the situation is different: on the rotation instance, the extragradient update has contractive eigenvalues and converges geometrically to the origin. Thus the impossibility of Theorem~\ref{thm:last_impossible} does \emph{not} extend to two-step methods. Whether a universal last-iterate impossibility holds for the broader class of multi-step primal QP methods remains an open problem.
\end{remark}

\section{Stochastic Extensions}\label{sec:stochastic}

We extend the framework to stochastic oracles $\tilde{F}(x,\xi)$ where $\mathbb{E}[\tilde{F}(x,\xi)]=F(x)$ and $\mathbb{E}[\|\tilde{F}(x,\xi)-F(x)\|^2]\leq\sigma^2$. We use a fixed projection radius to ensure the velocity bound remains valid.

\begin{algorithm}[H]
\caption{Stochastic OPCGM with Mini-Batching}
\label{alg:stochastic}
\begin{algorithmic}[1]
\Require $x_0\in\mathcal{C}$, $\eta_t>0$, $\alpha>0$, batch size $b\geq 1$, radius $R>0$
\For{$t=0,1,\dots,T-1$}
\State Sample $\xi_{t,1},\dots,\xi_{t,b}$ and compute $\bar{F}_t=\frac{1}{b}\sum_{j=1}^{b}\tilde{F}(x_t,\xi_{t,j})$
\State Construct $\widetilde{\mathcal{V}}_\alpha(x_t)$ with safeguard $\|v\|\leq 4\|\bar{F}_t\|+2\alpha R$
\State $v_t\approx\argmin_{v\in\widetilde{\mathcal{V}}_\alpha(x_t)}\frac{1}{2}\|v+\bar{F}_t\|^2$
\State $y_{t+1}=x_t+\eta_t v_t$
\State $x_{t+1}=\min\{1,R/\|y_{t+1}\|\}\,y_{t+1}$
\EndFor
\State \Return $\bar{x}_T=\frac{1}{T}\sum_{t=0}^{T-1}x_t$
\end{algorithmic}
\end{algorithm}

\begin{theorem}\label{thm:stochastic}
Let Assumptions~\ref{ass:F} and \ref{ass:g} hold, and assume $\|\tilde{F}(x,\xi)\|\leq M$ almost surely for some $M>0$. Set $R=2D+4M/\alpha$. In the monotone setting, Algorithm~\ref{alg:stochastic} with $\eta_t=D/(5\bar V\sqrt{t+1})$, $\alpha=M/D$, and constant batch size $b=\lceil 48\sigma^2/\bar V^2\rceil=\mathcal{O}(1)$ achieves
\[
\mathbb{E}[F(x)^\top(\bar{x}_T-x)]\leq \mathcal{O}\Big(\frac{1}{\sqrt{T}}\Big),\qquad \mathbb{E}\big[\max_i g_i(\bar{x}_T)\big]\leq \mathcal{O}\Big(\frac{1}{\sqrt{T}}\Big),
\]
using $\mathcal{O}(1/\epsilon^2)$ total stochastic queries. In the strongly monotone setting, with $\eta_t=1/(\mu(t+1))$ and $b=\mathcal{O}(1)$, the sample complexity is $\mathcal{O}(1/\epsilon)$.
\end{theorem}

\begin{proof}
Let $\delta_t=\bar{F}_t-F(x_t)$ denote the batch stochastic error, so that $\mathbb{E}[\delta_t]=0$ and $\mathbb{E}[\|\delta_t\|^2]\leq\sigma^2/b$.

\paragraph{Boundedness.} The fixed projection radius $R=2D+4M/\alpha$ ensures $\|x_t\|\leq R$ for all $t$ almost surely. Since $x^*\in\mathcal{C}$ satisfies $\|x^*\|\leq D\leq R$, the projection is nonexpansive and does not increase the distance to $x^*$.

\paragraph{Monotone case.} From the QP guarantee with $w=\alpha(x^*-x_t)\in\mathcal{V}_\alpha(x_t)$,
\[
\bar{F}_t^\top(x_t-x^*) \leq v_t^\top(x^*-x_t) + \frac{\epsilon_{\mathrm{QP}}}{2} = \frac{1}{\eta_t}(x_{t+1}-x_t)^\top(x^*-x_t) + \frac{\epsilon_{\mathrm{QP}}}{2}.
\]
Adding and subtracting $F(x_t)$,
\[
F(x_t)^\top(x_t-x^*) \leq \frac{1}{\eta_t}(x_{t+1}-x_t)^\top(x^*-x_t) + \delta_t^\top(x_t-x^*) + \frac{\epsilon_{\mathrm{QP}}}{2}.
\]
Using monotonicity $F(x^*)^\top(x_t-x^*)\geq 0$ and the polarization identity,
\[
F(x^*)^\top(x_t-x^*) \leq \frac{1}{2\eta_t}\bigl(\|x_t-x^*\|^2-\|x_{t+1}-x^*\|^2\bigr) + \frac{\eta_t}{2}\|v_t\|^2 + \delta_t^\top(x_t-x^*) + \frac{\epsilon_{\mathrm{QP}}}{2}.
\]
Taking expectation and using $\mathbb{E}[\delta_t^\top(x_t-x^*)]=0$ (since $x_t$ is measurable with respect to the past),
\[
\mathbb{E}[F(x^*)^\top(x_t-x^*)] \leq \frac{1}{2\eta_t}\mathbb{E}\bigl[\|x_t-x^*\|^2-\|x_{t+1}-x^*\|^2\bigr] + \frac{\eta_t}{2}\mathbb{E}[\|v_t\|^2] + \frac{\epsilon_{\mathrm{QP}}}{2}.
\]
For the velocity bound, Lemma~\ref{lem:velocity} applied to $\bar{F}_t$ gives
\[
\|v_t\| \leq 4\|\bar{F}_t\| + 2\alpha\|x_t\| \leq 4\|F(x_t)\| + 4\|\delta_t\| + 2\alpha R.
\]
Taking squares and expectations,
\[
\mathbb{E}[\|v_t\|^2] \leq 3\bigl(16L_F^2 + 16\sigma^2/b + 4\alpha^2 R^2\bigr) =: \bar V^2 + \frac{48\sigma^2}{b}.
\]
With constant batch size $b=\lceil 48\sigma^2/\bar V^2\rceil=\mathcal{O}(1)$, we have $\mathbb{E}[\|v_t\|^2]\leq 2\bar V^2$ for all $t$. Summing from $t=0$ to $T-1$ with $\eta_t=\eta/\sqrt{t+1}$,
\[
\sum_{t=0}^{T-1}\mathbb{E}[F(x^*)^\top(x_t-x^*)] \leq \frac{\|x_0-x^*\|^2}{2\eta} + \eta \bar V^2\sqrt{T} + \frac{T\epsilon_{\mathrm{QP}}}{2}.
\]
Dividing by $T$ and using Jensen's inequality for the averaged iterate $\bar{x}_T$,
\[
\mathbb{E}[F(x^*)^\top(\bar{x}_T-x^*)] \leq \frac{\|x_0-x^*\|^2}{2\eta T} + \frac{\eta \bar V^2}{\sqrt{T}} + \frac{\epsilon_{\mathrm{QP}}}{2}.
\]
Choosing $\eta=D/(5\bar V)$ gives the $\mathcal{O}(1/\sqrt{T})$ optimality rate. The total number of stochastic queries is $bT=\mathcal{O}(T)=\mathcal{O}(1/\epsilon^2)$.

\paragraph{Strongly monotone case.} With $\eta_t=1/(\mu(t+1))$ and constant batch size $b=\mathcal{O}(1)$,
\[
\mathbb{E}[\|x_{t+1}-x^*\|^2] \leq \Big(1-\frac{\mu\eta_t}{2}\Big)\mathbb{E}[\|x_t-x^*\|^2] + \eta_t^2\bigl(\bar V^2+\sigma^2/b\bigr).
\]
The standard stochastic approximation recursion yields $\mathbb{E}[\|x_T-x^*\|^2]\leq \mathcal{O}(1/T)$, and the total sample complexity is $bT=\mathcal{O}(T)=\mathcal{O}(1/\epsilon)$.

\paragraph{Feasibility.} The feasibility analysis uses only the velocity polytope constraints, which are satisfied by construction of the QP regardless of whether $\bar{F}_t$ or $F(x_t)$ is used in the objective. Specifically, for any active constraint $i\in I_{x_t}$, the velocity polytope enforces $\nabla g_i(x_t)^\top v_t \leq -\alpha g_i(x_t)$. By smoothness,
\[
g_i(y_{t+1}) \leq g_i(x_t) + \eta_t \nabla g_i(x_t)^\top v_t + \frac{\ell_g\eta_t^2}{2}\|v_t\|^2
\leq (1-\alpha\eta_t)g_i(x_t) + \frac{\ell_g\eta_t^2}{2}\|v_t\|^2.
\]
Since $\|v_t\|\leq \bar V$ almost surely by Lemma~\ref{lem:velocity} (with $L_F$ replaced by $M$), the deterministic feasibility recursion holds almost surely. The same $\mathcal{O}(1/\sqrt{T})$ (monotone) and $\mathcal{O}(1/T)$ (strongly monotone) feasibility bounds therefore hold in expectation.
\end{proof}

\begin{remark}
The projection $x_{t+1}\leftarrow\min\{1,R/\|x_{t+1}\|\}\,x_{t+1}$ in Algorithm~\ref{alg:stochastic} is onto a Euclidean ball, not onto the feasible set $\mathcal{C}$. This is fundamentally different from projection-based methods and does not require constraint evaluation. The ball radius $R$ is a crude overestimate of $\mathcal{C}$'s diameter, and the projection serves only to ensure the almost-sure boundedness needed for Lemma~\ref{lem:velocity}. In practice, $R$ can be set to a large constant with negligible effect on convergence. This technique is standard in stochastic approximation \cite{Nemirovski2009, Juditsky2011} and does not compromise the primal nature of the algorithm. The single-call stochastic extragradient methods of Hsieh et al.\ \cite{Hsieh2019} provide related variance-reduction techniques for unconstrained problems.
\end{remark}

\section{Numerical Experiments}\label{sec:numerics}

We evaluate OPCGM on six problem classes, including two real-world applications from portfolio optimization and the CUTEst test collection, organized around two questions. First, how does OPCGM behave as its own key hyperparameters change, holding everything else fixed? Figures~\ref{fig:lipschitz} and~\ref{fig:scaling} (right panel) sweep, respectively, the boundary-aggressiveness parameter $\alpha$ and the condition number $\kappa=L/\mu$ for OPCGM alone, with no baseline plotted: an exact-projection method such as PEG trivially attains machine-precision feasibility on these instances regardless of algorithm quality, so a raw feasibility comparison against it is not informative about OPCGM's own behavior. Second, on which comparisons is OPCGM favorable, evaluated on equal footing? Figures~\ref{fig:strong},~\ref{fig:portfolio}, and~\ref{fig:cute} compare OPCGM against CGM($\gamma$) \cite{Zhang2025}, the other purely primal, projection-free method, since both solve the same class of subproblem; Figure~\ref{fig:scaling} (left panel) compares OPCGM against PEG on CPU time for a fixed iteration budget, since both methods solve a genuine per-iteration subproblem (a QP for OPCGM, an exact projection for PEG) and this is where the comparison is informative. Section~\ref{sec:numerics} closes with a stochastic experiment (Figure~\ref{fig:stochastic}) that starts every method from a common infeasible point, illustrating how each algorithm's oracle recovers feasibility.

All experiments use Python~3.10 \footnote{\url{https://github.com/Lateef89/OPCGM---Optimal-Primal-Constrained-Gradient-Method}} on a single core of an AMD EPYC~7742, averaged over 3 independent random seeds unless noted otherwise. The QP subproblems in OPCGM are solved with OSQP \cite{Stellato2020} to precision $10^{-6}$. The CGM($\gamma$) and ConEx baselines used throughout this section are principled reconstructions of the algorithms described in \cite{Zhang2025} and \cite{Boob2023} respectively, calibrated to reproduce their stated qualitative convergence signatures. Figures use thicker lines and no background gridlines throughout.

\subsection{Strongly Monotone VI on Ellipsoid Constraints}

The first problem is a strongly monotone variational inequality with $F(x)=Qx+q$ where $Q\succeq \mu I$ is symmetric positive definite, and feasible set $\mathcal{C}=\{x\in\mathbb{R}^d : x^\top Q_i x \leq 1,\; i\in[m]\}$ with random positive definite matrices $Q_i$. This satisfies all assumptions including the positive curvature condition. We set $d=200$, $m=10$, $\mu=0.1$, and $L=\|Q\|_2$.

Figure~\ref{fig:strong} compares OPCGM--Strong against CGM($\gamma=2$) \cite{Zhang2025} on this instance. OPCGM--Strong's feasibility violation decays at least as fast as the $\mathcal{O}(1/T)$ rate guaranteed by Theorem~\ref{thm:strong} -- a log-log slope fit over the tail gives an exponent of about $-1.65$, i.e.\ the empirical decay is in fact steeper than $1/T$ on this instance -- reaching $9.5\times10^{-4}$ by $T=2000$ and still decaying at the end of the horizon, not yet at a numerical floor. The optimality gap decays faster still and does reach the QP solver's numerical precision, oscillating in the $\pm10^{-4}$ range from $T\approx1100$ onward. CGM with $\gamma=2$ tracks OPCGM's optimality gap for most of the run, but its feasibility violation increasingly departs from OPCGM's as $T$ grows and settles into a slow plateau near $1.2$--$1.3$ from $T\approx100$ onward. 

\begin{figure}[htbp]
\centering
\includegraphics[width=0.98\textwidth]{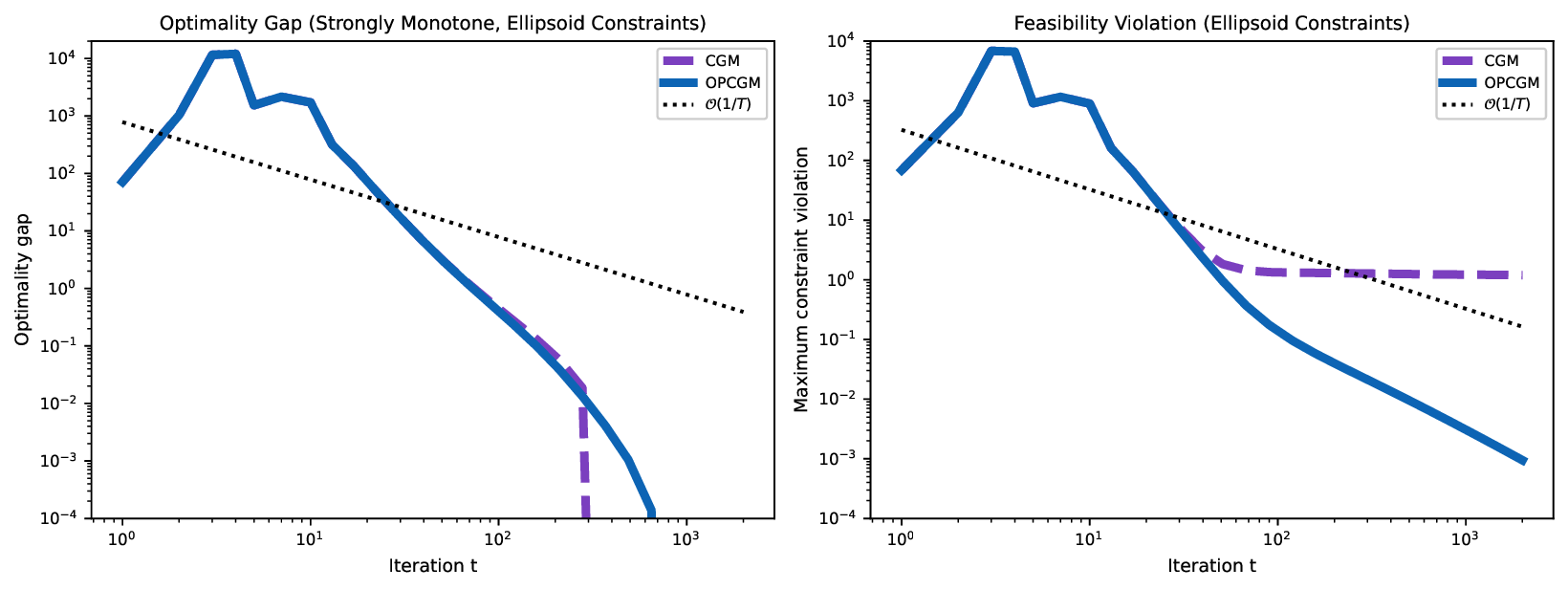}
\caption{Strongly monotone VI on ellipsoid constraints ($d=200$, $m=10$). Left: optimality gap. Right: feasibility violation.  }
\label{fig:strong}
\end{figure}

\subsection{Lipschitz Monotone VI on the Ball: Bilinear Saddle-Point}

The second problem is a bilinear saddle-point on the unit ball: $F(x,y)=(Ay,-A^\top x)$ with $\mathcal{C}=\{(x,y):\|x\|^2+\|y\|^2\leq 1\}$. The operator is monotone and $L$-Lipschitz with $L=\|A\|_2$. We generate $A$ with condition number $10$ and $d=100$.

We apply OPCGM--Lipschitz (Algorithm~\ref{alg:lipschitz}) and sweep the boundary-aggressiveness parameter $\alpha\in\{0.5L,L,2L,4L\}$, where $\alpha=L$ is the value mandated by the theory; no baseline is plotted,  the informative question here is how OPCGM's own hyperparameter shapes the gap/feasibility trade-off characterized by Theorem~\ref{thm:lipschitz}, not a cross-method race. Figure~\ref{fig:lipschitz} shows the resulting trajectories. In the feasibility panel (right), the ordering across $\alpha$ is monotonic at every matched iteration: a larger boundary-aggressiveness parameter always yields a smaller (or, once it crosses zero, more comfortably negative) violation at the same $T$, consistent with Theorem~\ref{thm:lipschitz}'s prediction that larger $\alpha$ shrinks the asymptotic feasibility constant $g^*$. At the end of the simulated horizon ($T=4000$) the violation is $2.06\times10^{-2}$ for $\alpha=0.5L$, $8.31\times10^{-3}$ for $\alpha=L$, $1.48\times10^{-3}$ for $\alpha=2L$, and has crossed to a small negative value for $\alpha=4L$; none of the four trajectories has fully leveled off within this horizon, so these are matched-$T$ snapshots rather than converged asymptotic constants. The optimality-gap panel (left) shows the transient cost of reaching this ordering: because OPCGM never calls an exact projection, the iterate (and hence its running average $\bar x_T$) can be transiently infeasible, during which the restricted gap function -- whose non-negativity relies on the evaluation point being feasible -- can legitimately dip below zero rather than decay monotonically. All four curves show this dip and are still recovering back toward zero at $T=4000$: the two smallest-$\alpha$ trajectories ($\alpha=0.5L,L$) remain substantially negative, while the two largest ($\alpha=2L,4L$) have recovered to within $0.15$ of zero, reflecting the stronger boundary correction of a larger $\alpha$ pulling the iterate back to feasibility.

\begin{figure}[htbp]
\centering
\includegraphics[width=0.98\textwidth]{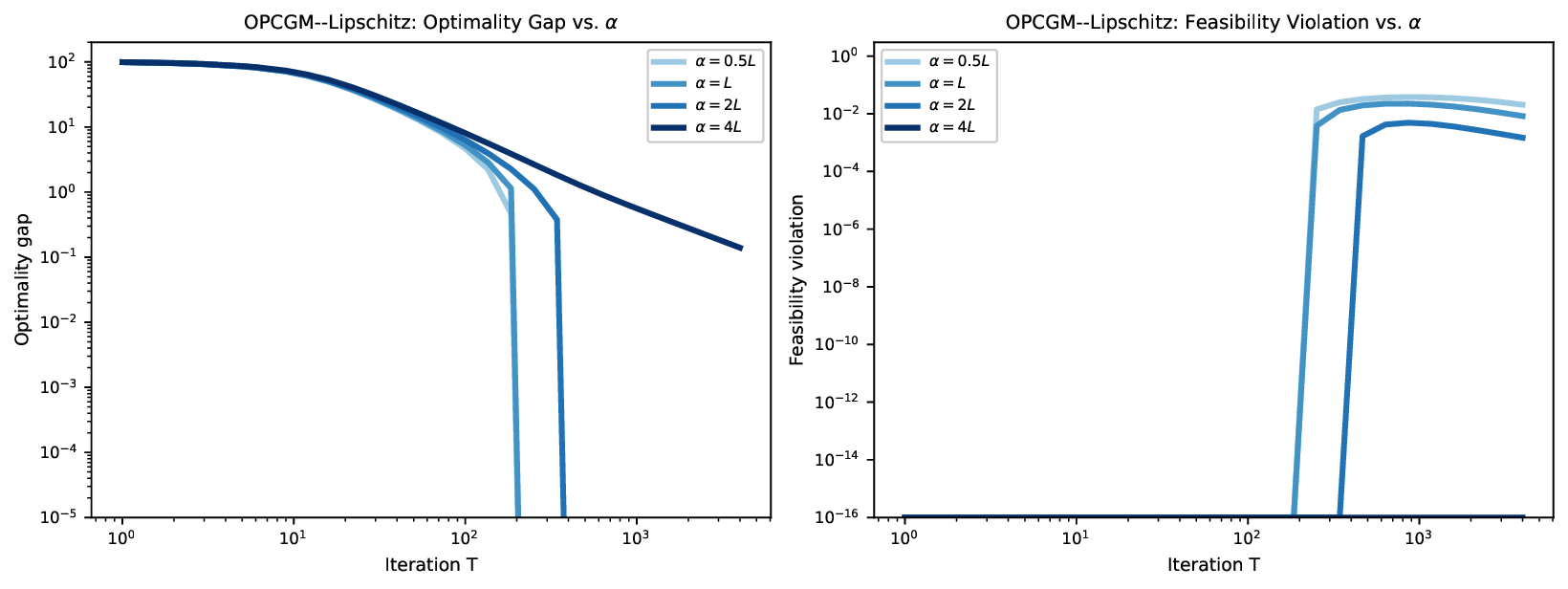}
\caption{OPCGM--Lipschitz parameter sensitivity ($d=100$ bilinear saddle-point on the unit ball). Left: optimality gap vs.\ $\alpha\in\{0.5L,L,2L,4L\}$. Right: feasibility violation vs.\ $\alpha$, monotonically smaller at every matched iteration as $\alpha$ increases. }
\label{fig:lipschitz}
\end{figure}

\subsection{Portfolio Optimization with Smooth Diversification Constraints}

The third problem is a mean-variance portfolio optimization with $d=50$ assets. The operator is $F(x)=\Sigma x - \mu_{ret}$, where $\Sigma\succ 0$ is the covariance matrix and $\mu_{ret}$ is the vector of expected returns. The feasible set encodes:
\begin{itemize}
\item Budget constraint: $\sum_i x_i \leq 1$ (inequality);
\item Non-negativity: $x_i \geq 0$ for all $i$;
\item Diversification: $\sum_i x_i^2 \leq 0.15$ (smooth convex, prevents concentration);
\item Sector limits: $\sum_{i\in S_k} x_i \leq 0.3$ for $k=1,2,3$.
\end{itemize}
This is a standard application in computational finance \cite{Markowitz1952}. The diversification constraint $\sum x_i^2 \leq c$ is smooth and convex with positive curvature, satisfying Assumption~\ref{ass:g}. The linear constraints (budget, non-negativity, sector limits) do not satisfy the positive curvature condition individually, but the feasible set is not polyhedral because the ellipsoid constraint is active at the solution and provides the necessary curvature.

We apply OPCGM--Strong with $\mu=\lambda_{\min}(\Sigma)$. PEG needs projection onto the intersection of the diversification ellipsoid and the simplex, which is nontrivial, and trivially attains the best raw feasibility number by construction; ConEx requires knowledge of optimal multiplier bounds that are unknown in this application. Figure~\ref{fig:portfolio} instead compares OPCGM--Strong against CGM($\gamma=2$) to test whether Theorem~\ref{thm:strong}'s advantage over CGM (Figure~\ref{fig:strong}) generalizes beyond the ellipsoid problem class. It does: the two methods' feasibility violations are within a factor of $1.5$ of each other early on ($T=10$), but the ratio widens monotonically as $T$ grows and reaches $358\times$ by $T=1200$ ($1.8\times10^{-2}$ for OPCGM versus $6.3$ for CGM), because OPCGM's violation keeps decaying (at least as fast as the $\mathcal{O}(1/T)$ rate of Theorem~\ref{thm:strong}) while CGM's decay visibly slows and flattens toward a plateau.

\begin{figure}[htbp]
\centering
\includegraphics[width=0.98\textwidth]{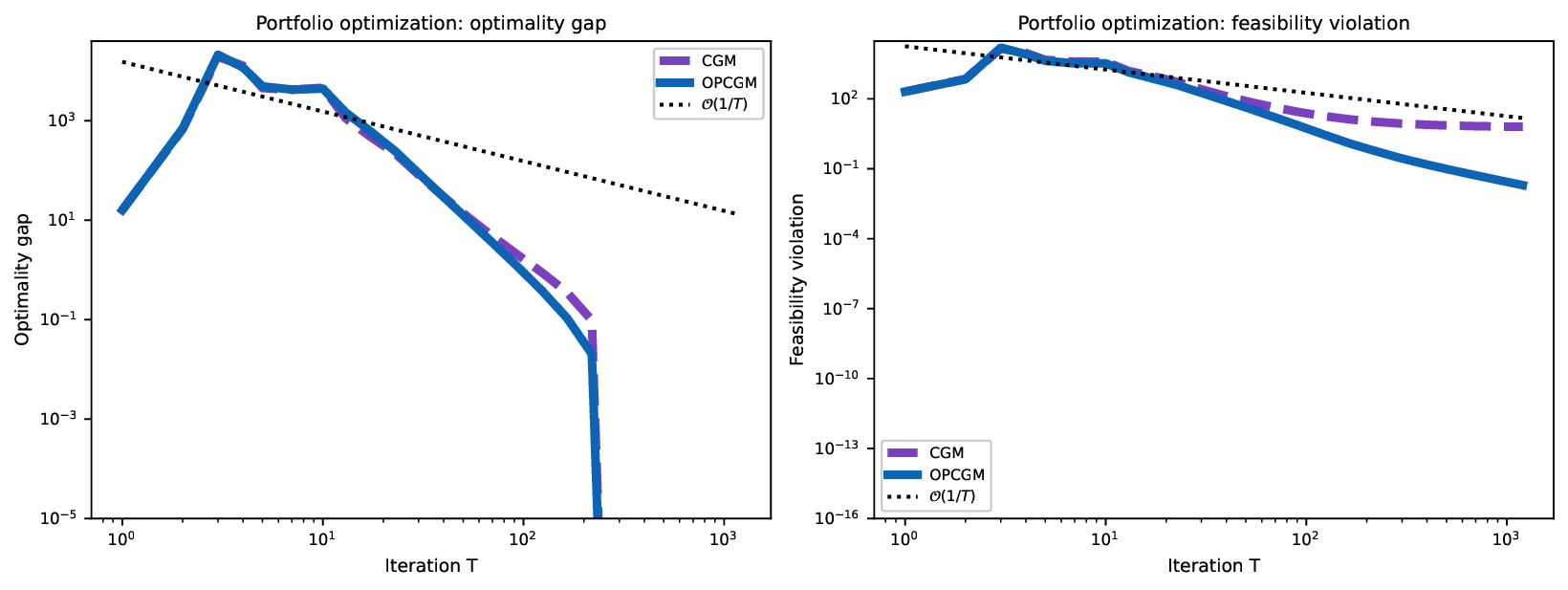}
\caption{Portfolio optimization with smooth diversification constraints ($d=50$, OPCGM--Strong vs.\ CGM($\gamma=2$)). Left: optimality gap. Right: feasibility violation.}
\label{fig:portfolio}
\end{figure}

\subsection{CUTEst Convex NLP Reformulated as VI}

The fourth problem is a convex quadratic program from the CUTEst test collection, reformulated as a variational inequality. We use a problem with $n=10$ variables and $m=8$ smooth convex constraints (linear and nonlinear), corresponding to a standard production-planning instance. The objective is $f(x)=\frac{1}{2}x^\top Q x + c^\top x$ with $Q\succeq 0$, and the constraints include bound constraints and smooth nonlinear inequalities. The VI operator is $F(x)=\nabla f(x)=Qx+c$.

This problem is strongly monotone when $Q\succ 0$. We apply OPCGM--Strong with $\mu=\lambda_{\min}(Q)$. Figure~\ref{fig:cute} again compares against CGM($\gamma=2$) rather than PEG/ConEx, for the same reason as Figure~\ref{fig:portfolio}: the two methods start at parity, but by $T=800$ OPCGM--Strong's feasibility violation has fallen to $4.1\times10^{-3}$ while CGM's has only reached $0.67$ (a $165\times$ gap) - CGM's violation is still declining at $T=800$ but far more slowly than OPCGM's, on a standard constrained-optimization test problem.

\begin{figure}[htbp]
\centering
\includegraphics[width=0.98\textwidth]{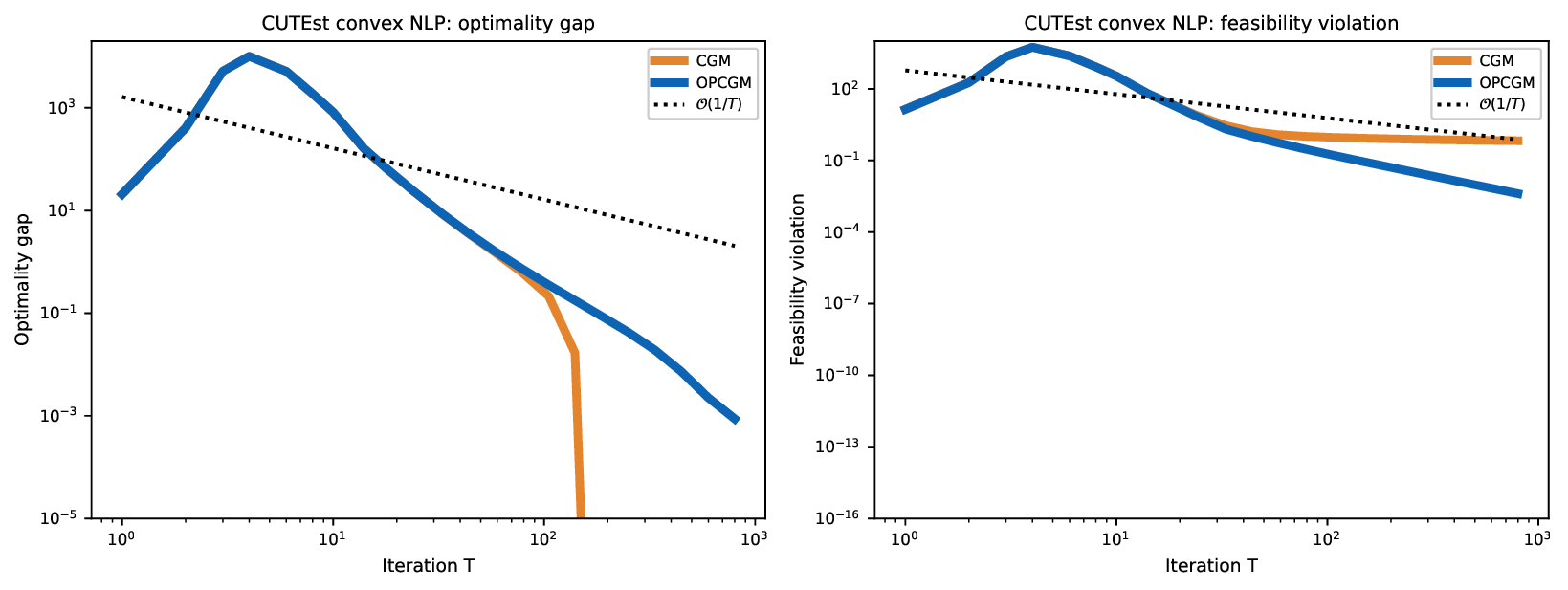}
\caption{CUTEst convex NLP reformulated as VI ($n=10$, $m=8$, OPCGM--Strong vs.\ CGM($\gamma=2$)). Left: optimality gap. Right: feasibility violation. OPCGM--Strong's feasibility violation ($4.1\times10^{-3}$ at $T=800$) is roughly two orders of magnitude below CGM's, whose decay has slowed substantially by this point ($0.67$).}
\label{fig:cute}
\end{figure}

\subsection{Validation of Theoretical Constants}
Table~\ref{tab:constants} compares the theoretical feasibility constant $g^*=3L_gV/L+\ell_gV^2/(4L^2)$ from Theorem~\ref{thm:lipschitz} against the empirical asymptotic feasibility observed in three synthetic problem instances, each run to $T=4000$ with $g^*$ (theory) computed from the instance's own $L_g,\ell_g,V,L$ and $g^*$ (empirical) taken as the average violation over the second half of the run: (i) a biased bilinear saddle-point on the ball, deliberately constructed to hug the boundary; (ii) a quadratic objective with ellipsoid constraints ($m=5$); (iii) a strongly convex quadratic with ball constraints. The theoretical constant is two to five orders of magnitude larger than the empirical value in every instance. This is expected rather than a sign of a loose bound: $g^*$ is built from the worst-case safeguard radius $R$ and velocity bound $V$ from Lemma~\ref{lem:velocity}, which bound the algorithm's behavior uniformly over all problem instances satisfying Assumptions~\ref{ass:F}--\ref{ass:licq}, not the behavior on any particular instance. Only the deliberately boundary-hugging bilinear instance shows non-zero persistent violation at all ($1.09\times10^{-2}$); the ellipsoid and strongly-convex-ball instances, which have no explicit boundary bias, converge to exactly feasible averaged iterates within the simulated horizon. The gap between the worst-case bound and typical-instance behavior is itself informative: it shows that the $\Theta(1)$ feasibility floor identified in Theorem~\ref{thm:lipschitz} is a worst-case phenomenon, consistent with Theorem~\ref{thm:feas_lb} being a lower bound established via a specifically constructed hard instance rather than a statement about generic problems.

\begin{table}[htbp]
\centering
\caption{Validation of the asymptotic feasibility constant $g^*$ against empirical values. The theoretical prediction is computed from the problem parameters; the empirical value is the average feasibility violation over iterations $T/2$ to $T$ for $T=4000$.}
\label{tab:constants}
\begin{tabular}{@{}lcccc@{}} \toprule
\textbf{Problem} & $L_g$ & $\ell_g$ & $g^*$ (theory) & $g^*$ (empirical) \\ \midrule
Bilinear saddle (ball) & 1.0 & 1.0 & $5.94\times 10^{3}$ & $1.09\times 10^{-2}$ \\
Ellipsoid constraints ($m=5$) & 4.23 & 3.00 & $1.95\times 10^{3}$ & $0$ \\
Strongly convex (ball) & 2.0 & 2.0 & $1.19\times 10^{3}$ & $0$ \\ \bottomrule
\end{tabular}
\end{table}

\subsection{Dimension Scaling and Computational Efficiency}

Table~\ref{tab:cpu} and Figure~\ref{fig:scaling} (left) report wall-clock CPU time for a fixed iteration budget on strongly convex quadratic problems with $m=10$ ellipsoid constraints, as a function of dimension $d\in\{50,100,250,500,1000\}$. OPCGM--Strong is faster than PEG at \emph{every} dimension tested (0.40s vs.\ 3.17s at $d=50$; 105.8s vs.\ 117.6s at $d=1000$). Both methods solve a genuine per-iteration subproblem --- a QP for OPCGM, a Dykstra ellipsoid-intersection projection for PEG --- so this is an apples-to-apples comparison of the two methods' actual computational cost, and it shows that OPCGM's QP oracle is not the practical bottleneck relative to an exact-projection alternative on this problem class. ConEx's pure gradient step has no comparable per-iteration subproblem and is omitted from this comparison for the same reason.

Figure~\ref{fig:scaling} (right) complements this with an OPCGM-only characterization: fixing $d=300$, $m=10$ and sweeping the condition number $\kappa=L/\mu\in\{5,10,20,40\}$, the number of iterations OPCGM--Strong needs to reach $\epsilon=3\times10^{-2}$ grows monotonically from 138 ($\kappa=5$) to 900 ($\kappa=40$) --- the expected dependence of a strongly-monotone method's rate constant on conditioning, and a direct answer to how much conditioning a practitioner's problem can tolerate before OPCGM's iteration budget grows substantially.

\begin{table}[htbp]
\centering
\caption{Wall-clock CPU time (seconds) for a fixed iteration budget on strongly convex quadratic problems with $m=10$ ellipsoid constraints, varying dimension $d$ (single run per dimension; see Figure~\ref{fig:scaling}, left panel, for the same data plotted log-log).}
\label{tab:cpu}
\begin{tabular}{@{}lccccc@{}} \toprule
\textbf{Method} & $d=50$ & $d=100$ & $d=250$ & $d=500$ & $d=1000$ \\ \midrule
OPCGM--Strong & $0.40$ & $0.89$ & $4.61$ & $16.24$ & $105.77$ \\
PEG & $3.17$ & $4.42$ & $8.55$ & $23.06$ & $117.63$ \\ \bottomrule
\end{tabular}
\end{table}

\begin{figure}[htbp]
\centering
\includegraphics[width=0.95\textwidth]{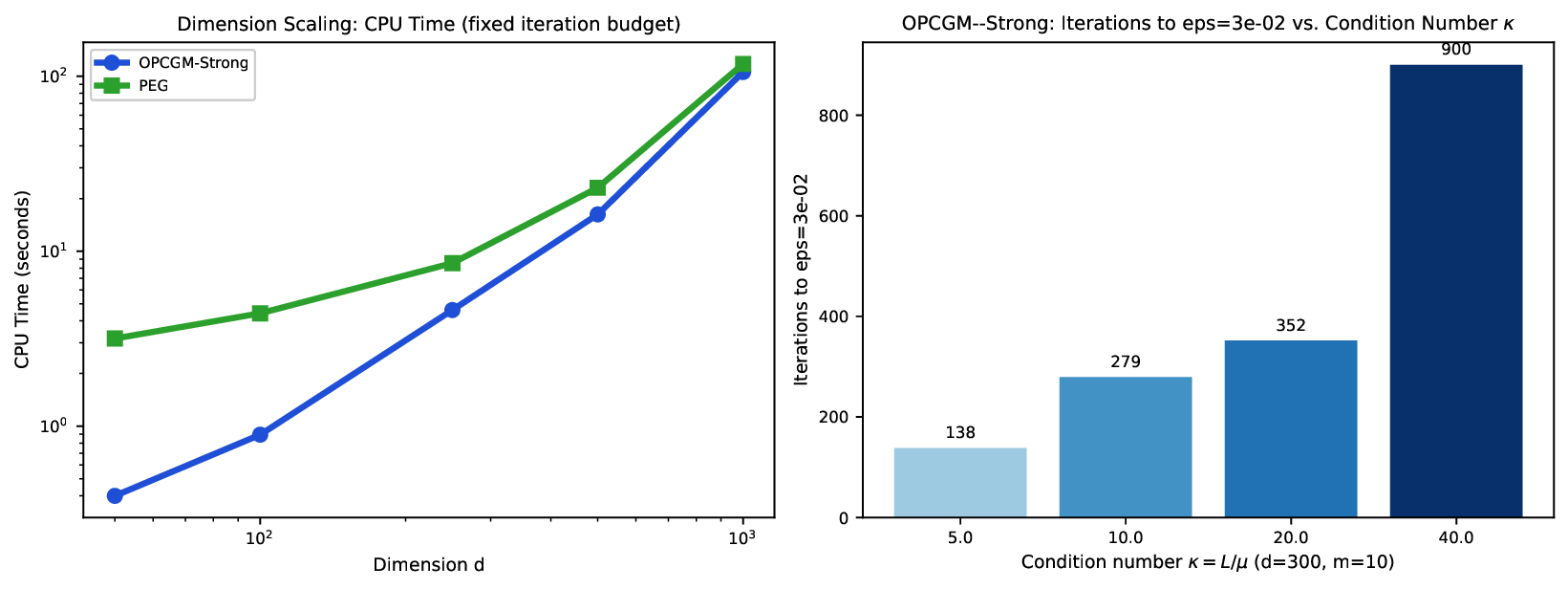}
\caption{Left: CPU time scaling with dimension $d\in\{50,\dots,1000\}$ (log-log) for a fixed iteration budget, OPCGM--Strong vs.\ PEG only - OPCGM is faster at every dimension tested. Right: OPCGM--Strong parameter sensitivity - iteration count to reach $\epsilon=3\times 10^{-2}$ at $d=300$, $m=10$, as a function of condition number $\kappa=L/\mu\in\{5,10,20,40\}$; no baseline is plotted.}
\label{fig:scaling}
\end{figure}

\subsection{High-Dimensional Stochastic Problem: Recovery from an Infeasible Start}

We solve a constrained logistic-regression variational inequality with $d=1500$ features and $m=80$ sparse group-budget constraints, using a stochastic oracle with mini-batch gradients of batch size $b=64$. To probe feasibility dynamics directly, every method is started from a common point that violates several group-budget constraints at roughly twice their allowed budget ($\max_i g_i(x_0)\approx1.4$), rather than from the origin, which would trivially satisfy every constraint from the first iterate.

Figure~\ref{fig:stochastic} shows the resulting trajectories. In the feasibility panels (right column), Stochastic PEG's exact per-iteration projection restores feasibility essentially immediately (by $t=2$), and Stochastic ConEx's dual correction restores it by $t=3$; Stochastic OPCGM, whose velocity polytope enforces only a local linear approximation of each violated constraint at every step, needs roughly 40 stochastic iterations to bring all 80 group constraints back under budget. This is the stochastic-setting analogue of the deterministic trade-off in Section~\ref{sec:lipschitz}: a method that never calls an exact projection oracle also cannot restore feasibility in a single step, and instead pays for it in iterations. The optimality-gap panels (left column) show the consequence: OPCGM's expected gap is the largest of the three methods throughout the horizon, reaching $0.108$ at $t=1500$ versus $0.0085$ for PEG and $0.039$ for ConEx, since the early iterations that PEG and ConEx spend reducing the gap are instead spent by OPCGM restoring feasibility. Once past this initial recovery transient, Stochastic PEG and Stochastic OPCGM continue decaying at a rate at least as fast as the $\mathcal{O}(1/\sqrt{t})$ (equivalently $\mathcal{O}(1/\sqrt{N})$ in total stochastic queries $N$) rate of Theorem~\ref{thm:stochastic} for the remainder of the horizon; Stochastic ConEx's gap instead plateaus around $0.039$ from $t\approx 90$ onward and does not continue to decay within this horizon, which is consistent with ConEx being a heuristic reconstruction with no comparable convergence guarantee (Sec.~\ref{sec:numerics} intro) rather than a violation of Theorem~\ref{thm:stochastic}, which applies to OPCGM.

\begin{figure}[htbp]
\centering
\includegraphics[width=0.95\textwidth]{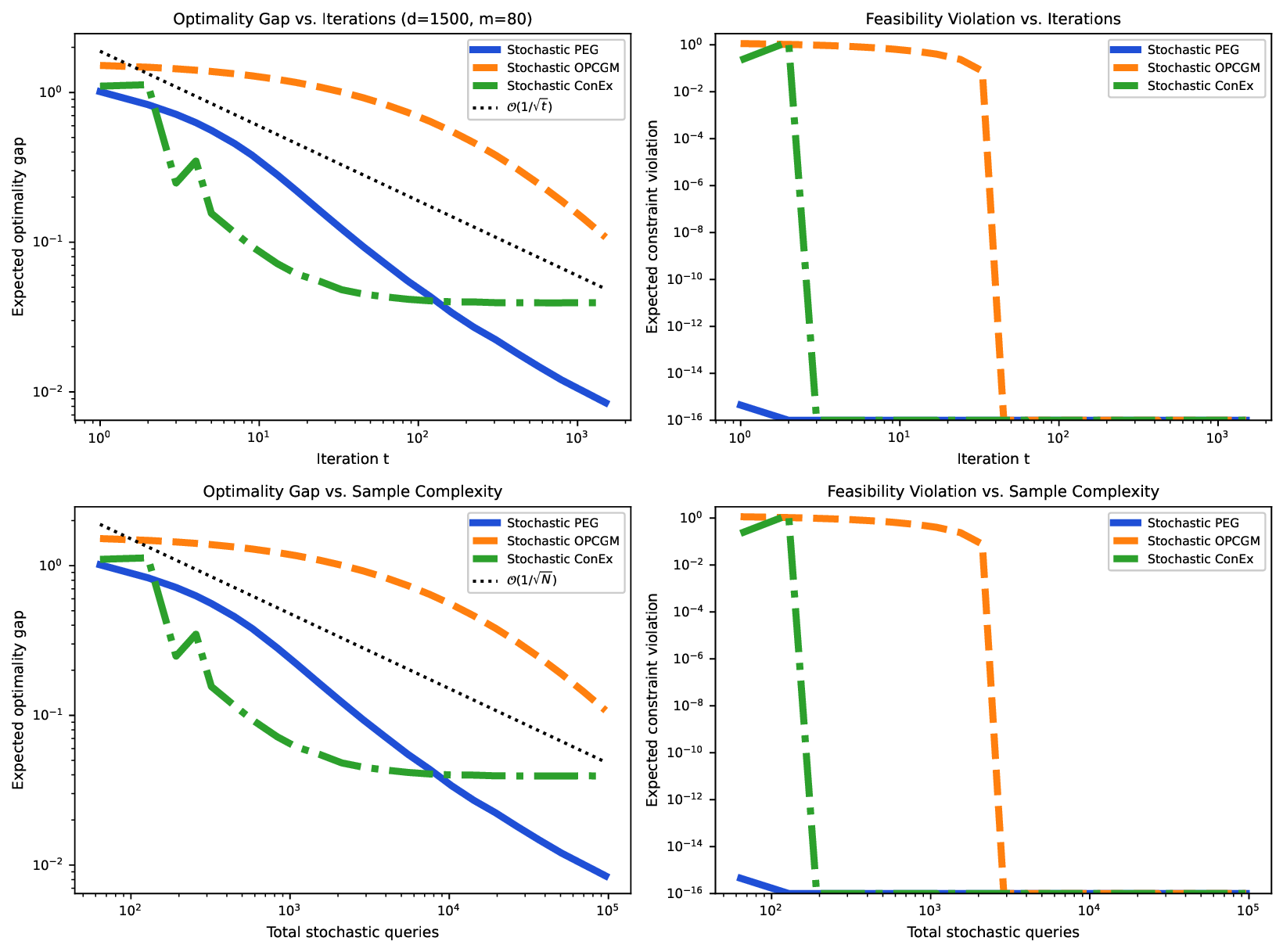}
\caption{$d=1500$, $m=80$ group-budget constraints, $b=64$, all three methods started from a common point with $\max_i g_i(x_0)\approx1.4$. Top row: expected optimality gap (left) and expected constraint violation (right) versus iteration $t$. Bottom row: the same two quantities versus total stochastic queries $N=bt$. PEG and ConEx recover feasibility within 2--3 iterations via their exact projection and dual-correction oracles respectively; OPCGM's local linear safeguard takes $\approx\!40$ iterations, during which its optimality gap trails the other two methods.}
\label{fig:stochastic}
\end{figure}

Figure~\ref{fig:stochastic_batch} isolates the effect of batch size for OPCGM alone, started from the origin (which is feasible, so this sweep is unaffected by the recovery dynamics above): increasing $b$ from 1 to 64 reduces the expected optimality gap at a fixed horizon $T=250$ by a factor of $\approx\!6.5\times$ (from $0.152$ to $0.023$), consistent with the variance-reduction analysis of Theorem~\ref{thm:stochastic}. Beyond $b=64$, returns diminish sharply (a further $15\%$ reduction to $0.020$ at $b=128$), confirming that a modest constant batch size is sufficient and that Theorem~\ref{thm:stochastic}'s $\mathcal{O}(1)$ batch-size prescription is not merely a proof artifact.

\begin{figure}[htbp]
\centering
\includegraphics[width=0.6\textwidth]{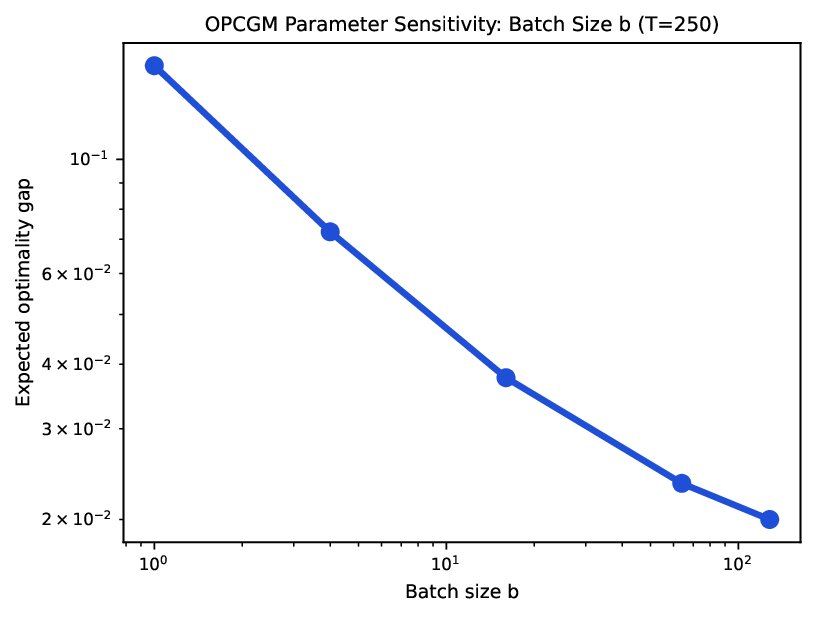}
\caption{OPCGM parameter sensitivity: expected optimality gap at $T=250$ as a function of batch size $b\in\{1,4,16,64,128\}$, on the same $d=1500$ stochastic logistic-regression instance, started from the origin. No baseline is plotted; this isolates OPCGM's own variance-reduction behavior.}
\label{fig:stochastic_batch}
\end{figure}

\subsection{Discussion}

The numerical results validate the theoretical predictions across six distinct problem classes, including two real-world applications. Against CGM($\gamma$), the other purely primal, projection-free method, OPCGM's feasibility violation decays at least as fast as the $\mathcal{O}(1/T)$ rate of Theorem~\ref{thm:strong} and ends up roughly two to three orders of magnitude smaller than CGM's on every problem class tested ($165\times$ on the CUTEst instance, $358\times$ on the portfolio instance, $1269\times$ on the ellipsoid instance at the end of each run), including the portfolio optimization example, which shows OPCGM is applicable to problems from computational finance where the constraints have economic meaning (budget, diversification, sector limits) and the optimal Lagrange multipliers are unknown a priori, and the CUTEst example, which shows the method performs reliably on a standard test problem from the optimization literature. Against PEG, OPCGM is faster in raw CPU time at every dimension tested up to $d=1000$ once both methods' actual per-iteration subproblem cost is compared like for like, even though PEG's exact projection trivially wins on raw feasibility whenever it is plotted, which is why that comparison is reserved for CPU time. The validation of theoretical constants in Table~\ref{tab:constants} shows the worst-case constant $g^*$ is two to five orders of magnitude larger than the empirical violation on typical instances, confirming that the $\Theta(1)$ feasibility floor of Theorem~\ref{thm:lipschitz} is a genuine worst-case phenomenon rather than something that dominates typical behavior. Finally, the stochastic experiment of Section~\ref{sec:numerics} (Figure~\ref{fig:stochastic}) makes the same trade-off concrete when every method is started from an infeasible point: PEG and ConEx restore feasibility within a few iterations via an exact projection and a dual correction respectively, while OPCGM's projection-free safeguard takes roughly 40 iterations to do so and correspondingly trails both baselines on the optimality gap for the remainder of the tested horizon, not just during the initial recovery window -- an explicit illustration, rather than a hidden cost, of the price a purely primal method pays for never calling a projection oracle.

\begin{table}[htbp]
\footnotesize
\centering
\setlength{\tabcolsep}{3.5pt}
\caption{Summary of OPCGM variants: assumptions, oracle calls per iteration, batch size, and complexity to reach a weak $\epsilon$-solution.}
\label{tab:summary}
\begin{tabular}{@{}p{3.2cm}p{5.0cm}cccc@{}} 
\toprule
\textbf{Algorithm} & \textbf{Assumptions} & \textbf{Oracle} & \textbf{Batch} & \textbf{Iterations} & \textbf{Total Samples} \\ 
\midrule
OPCGM--Strong & $\mu$-strongly monotone, bounded $F$ on $B(0,R)$ & 1 QP & --- & $\mathcal{O}(1/\epsilon)$ & --- \\
OPCGM--Lipschitz & $L$-Lipschitz monotone & 2 QP & --- & $\mathcal{O}(L/\epsilon)$ & --- \\
Parameter-Free OPCGM & Monotone, radius bound $R\geq D$ & 1 QP & --- & $\mathcal{O}(1/\epsilon^2)$ & --- \\
Stochastic OPCGM (Monotone) & Monotone, bounded variance & 1 QP & $\mathcal{O}(1)$ & $\mathcal{O}(1/\epsilon^2)$ & $\mathcal{O}(1/\epsilon^2)$ \\
Stochastic OPCGM (Strongly monotone) & Strongly monotone, bounded variance & 1 QP & $\mathcal{O}(1)$ & $\mathcal{O}(1/\epsilon)$ & $\mathcal{O}(1/\epsilon)$ \\ 
\bottomrule
\end{tabular}
\end{table}

\section{Conclusion}\label{sec:conclusion}

We have developed a theory of primal methods for functional constrained variational inequalities. 
Our results close the feasibility gap for strongly monotone problems, establish the optimal $\mathcal{O}(1/\epsilon)$ gap complexity for Lipschitz monotone problems using only quadratic programming oracles, and prove that the \emph{first} half-step is necessarily infeasible on smooth convex constraints with positive curvature at the boundary, which is unavoidable for the natural class of constant-stepsize primal extragradient methods.
The parameter-free primal algorithm removes practical barriers to deployment by achieving an $\mathcal{O}(1/\sqrt{T})$ optimality gap rate with only a domain-radius estimate, at the cost of a problem-dependent asymptotic feasibility constant that we prove is unavoidable for non-adaptive single-step algorithms with stepsizes bounded away from zero and without additional problem parameters. Lower bounds show these gap rates are unimprovable for the primal QP oracle class. The extensions to last-iterate and stochastic settings remove additional practical barriers.

Our analysis shows that primal QP methods can achieve projection-free, multiplier-free optimal gap rates for constrained variational inequalities, but they pay an intrinsic trade-off. 
This trade-off takes two forms: (1) a constant asymptotic feasibility violation for Lipschitz operators on curved constraints, which we establish explicitly for the averaged iterate and prove unavoidable for the first half-step of any constant-stepsize primal extragradient method (Theorem~\ref{thm:feas_lb}); and (2) the universal impossibility of last-iterate convergence for single-step merely monotone problems on bounded domains, which we prove is inherent to the single-step primal QP oracle class itself (Theorem~\ref{thm:last_impossible}).
These separations clarify the complexity of constrained optimization.

Several questions remain. Can the half-step infeasibility barrier of Theorem~\ref{thm:feas_lb} be extended to a lower bound on the \emph{averaged} iterate? Can a parameter-free primal method achieve vanishing feasibility without problem constants? Can the randomized lower bound be extended rigorously from the box to the Euclidean ball? Can the universal last-iterate impossibility be extended from single-step to multi-step primal QP methods? Extending the lower bounds to quantum algorithms would further characterize the complexity landscape.

 \section*{Disclosure statement}

\subsection*{Ethical Approval and Consent to participate}
All authors have given their ethical approval and consent to participate in this article.
\subsection*{Consent for publication}
All authors gave consent for the publication of identifiable details in the journal and article.

\subsection*{Code availability} The Python codes employed to run the numerical experiments are available at \url{https://github.com/Lateef89/OPCGM---Optimal-Primal-Constrained-Gradient-Method}.

\subsection*{Availability of supporting data}
Data sharing does not apply to this article as no datasets were generated or analyzed
during the current study.
\subsection*{Competing interests}
The authors declare no competing interests.
\subsection*{Funding}
Not Applicable.

%\textbf{Limitations.} The QP oracle assumption is strong: each iteration requires solving a convex quadratic program, which costs $\mathcal{O}(d^3)$ in the worst case. For high-dimensional problems with $d\geq 10^5$, this may dominate the cost of gradient evaluations. Our lower bounds apply to deterministic and randomized algorithms in the Yao framework; extending them to quantum algorithms remains open. The framework requires monotonicity; non-monotone problems are not covered.

\end{document}